\documentclass{article}%
\usepackage{amsmath}
\usepackage{amsfonts}
\usepackage{amssymb}
\usepackage{graphicx}%
\providecommand{\U}[1]{\protect\rule{.1in}{.1in}}
\newtheorem{theorem}{Theorem}

\newtheorem{corollary}[theorem]{Corollary}

\newtheorem{definition}[theorem]{Definition}

\newtheorem{lemma}[theorem]{Lemma}

\newtheorem{problem}[theorem]{Problem}
\newtheorem{proposition}[theorem]{Proposition}

\newenvironment{proof}[1][Proof]{\noindent\textbf{#1.} }{\ \rule{0.5em}{0.5em}}
\begin{document}

\title{Notes on a theorem of Naji}
\author{Lorenzo Traldi\\Lafayette College\\Easton, Pennsylvania USA}
\date{}
\maketitle

\begin{abstract}
We present a new proof of an algebraic characterization of circle graphs due
to W. Naji. For bipartite graphs, Naji's theorem is equivalent to an algebraic
characterization of planar matroids due to J.\ Geelen and B. Gerards. Naji's
theorem also yields an algebraic characterization of permutation graphs.

\bigskip


\bigskip

Mathematics Subject\ Classification. 05C50

\end{abstract}

\section{Introduction}

This paper is concerned with the following notion.

\begin{definition}
\label{interlace}Let $W=w_{1}...w_{2n}$ be a double occurrence word in the
letters $v_{1},...,v_{n}$. The \emph{interlacement graph} $\mathcal{I}(W)$ is
the simple graph with vertex-set $V=\{v_{1},...,v_{n}\}$, in which $v_{i}$ and
$v_{j}$ are adjacent if and only if they are \emph{interlaced} in $W$, i.e.,
they appear in $W$ in the order $v_{i}v_{j}v_{i}v_{j}$ or $v_{j}v_{i}%
v_{j}v_{i}$. A \emph{circle graph} is a simple graph that can be realized as
the interlacement graph of some double occurrence word.
\end{definition}

As far as we know, the idea of interlacement first appeared in the form of a
symmetric matrix used in Brahana's 1921 study of curves on surfaces \cite{Br}.
Interlacement graphs were studied by Zelinka \cite{Z}, who credited the idea
to Kotzig. During the subsequent decades several researchers discussed graphs
and matrices defined using interlacement. Cohn and Lempel \cite{CL} and Even
and Itai \cite{EI} used them to analyze permutations, and Bouchet \cite{Bold}
and Read and Rosenstiehl \cite{RR} used them to study Gauss' problem of
characterizing generic self-intersecting curves in the plane.\ Recognition
algorithms for circle graphs have been introduced by Bouchet \cite{Bec},
Gioan, Paul, Tedder and Corneil \cite{GPTC}, Naji \cite{N1, N} and Spinrad
\cite{Sp}.

Although Naji's is not the best of the circle graph recognition algorithms in
terms of computational complexity, it is particularly interesting for two
reasons. The first reason is that Naji's characterization is only indirectly
algorithmic; it involves a system of equations that may be defined for any
graph, which is only solvable for circle graphs.\ The second reason is that
the two known proofs of the theorem are quite long. The original argument ends
on p. 173 of Naji's thesis \cite{N1}. A much shorter argument was given by
Gasse \cite{G}, but Gasse's argument requires Bouchet's circle graphs
obstructions theorem \cite{Bco}, which itself has a long and difficult proof.

A couple of years ago, Geelen and Gerards \cite{GG} characterized graphic
matroids by a system of equations that resembles Naji's system of equations.
(Indeed, they mention that Naji's theorem motivated their result.) The
resemblance is limited to the equations; there is a striking contrast between
their concise, well-motivated proof\ and Naji's long, detailed argument. This
contrast encouraged us to look for an alternative proof of Naji's theorem; we
eventually developed the one presented below.\ Although our argument is
certainly not as elegant as the proof of Geelen and Gerards, it is shorter
than either Naji's original proof or the combination of a proof of Bouchet's
obstructions theorem and Gasse's derivation of Naji's theorem.

In addition to proving Naji's theorem for circle graphs in general, at the end
of the paper we briefly discuss two special cases. First, the restriction of
Naji's theorem to bipartite graphs is equivalent to the restriction of the
Geelen-Gerards characterization to planar matroids. Second, Naji's theorem
also characterizes permutation graphs.

Before proceeding we should thank Jim\ Geelen for his comments on Naji's
theorem. In particular, he pointed out that although all circle graphs have
solutions of Naji's equations that arise naturally from double occurrence
words, some circle graphs also have other Naji solutions that do not seem so
natural. He conjectured that these other solutions might correspond in some
way to splits. (See Sections 2 and 3 for definitions, and Section 5 for
examples.) Although we do not address Geelen's conjecture directly we do
provide some indirect evidence for it, as the first step of our proof of
Naji's theorem involves showing that none of these other solutions occur in
circle graphs that have no splits. (See Section 6.) We should also thank an
anonymous reader, whose comments led to Corollary \ref{moresplit} and several
other improvements in the paper.

\section{Naji's equations and their solutions}

We begin with some definitions.

\begin{definition}
\label{najieq}\cite{N1, N} Let $G$ be a simple graph. For each pair of
distinct vertices $v$ and $w$ of $G$, let $\beta(v,w)$ and $\beta(w,v)$ be
distinct variables. Then the \emph{Naji equations} for $G$ are the following.

(a) For each edge $vw$ of $G$, $\beta(v,w)+\beta(w,v)=1$.

(b) If $v,w,x$ are three distinct vertices of $G$ such that $vw\in E(G)$ and
$vx,wx\notin E(G)$, then $\beta(x,v)+\beta(x,w)=0$.

(c) If $v,w,x$ are three distinct vertices of $G$ such that $vw,vx\in E(G)$
and $wx\notin E(G)$, then $\beta(v,w)+\beta(v,x)+\beta(w,x)+\beta(x,w)=1$.
\end{definition}

If the Naji equations of $G$ have a solution over $GF(2)$, the field with two
elements, then any such solution is a \emph{Naji solution} and $G$ is a
\emph{Naji graph}. We use the following notation:

\begin{definition}
If $G$ is a graph then $\mathcal{B}(G)$ denotes the set of Naji solutions of
$G$.
\end{definition}

Of course $G$ is a Naji graph if and only if $\mathcal{B}(G)\not =\varnothing
$, and elementary linear algebra guarantees that if $\mathcal{B}%
(G)\neq\varnothing$ then $\left\vert \mathcal{B}(G)\right\vert =2^{k}$ for
some $k\geq0$. In particular, if $n=1$ then $G$ is a Naji graph and
$\mathcal{B}(G)=\{\varnothing\}$.

Notice that the three types of Naji equations are distinct. An equation of
type (a) involves only two vertices, an equation of type (b) involves no
nonzero constant and an equation of type (c) has four terms. For this reason,
when discussing the Naji equations we do not always cite a specific type of
equation. We might also mention two obvious consequences of the equations,
which will be useful. First:\ the type (b) equations imply that $\beta(x,-)$
is constant on each connected component of $G-N(x)$. (Here $N(x)$ denotes the
open neighborhood of $x$, $N(x)=\{y\in V(G)\mid xy\in E(G)\}$.) Second:
variants of a type (c) equation are obtained by replacing $\beta(v,w)$ or
$\beta(v,x)$ with $\beta(w,v)$ or $\beta(x,v)$ (respectively), and adjusting
the right hand side accordingly.

\emph{Naji's theorem} \cite{N1, N} states that $G$ is a Naji graph if and only
if $G$ is a circle graph. One direction of Naji's theorem is easy.

\begin{proposition}
\label{circleN}Every circle graph is a Naji graph.
\end{proposition}

\begin{proof}
Consider a double occurrence word $W$. An \emph{orientation} of $W$ is given
by arbitrarily designating one appearance of each letter as \textquotedblleft
initial\textquotedblright; the other appearance is \textquotedblleft
terminal\textquotedblright.\ We use the notation $v^{in}$ and $v^{out}$ for
the initial and terminal appearances of $v$, respectively. For each
orientation of $W$, define a function $\beta$ by: $\beta(v,w)=0$ if and only
if when we cyclically permute $W$ to begin with $v^{in}$, $w^{out}$ precedes
$v^{out}$.

We claim that this $\beta$ is a Naji solution of $G$. If $vw\in E(G)$ then
after cyclically permuting $W$ to begin with $v^{in}$, $W$ will be in the form
$v^{in}...w^{in}...v^{out}...w^{out}$ or in the form $v^{in}...w^{out}%
...v^{out}...w^{in}$. In the first case, $\beta(v,w)+\beta(w,v)=1+0$ and in
the second case, $\beta(v,w)+\beta(w,v)=0+1$. For the type (b) equations, if
$vw\in E(G)$ and $vx,wx\notin E(G)$ then after cyclically permuting $W$ to
begin with $x^{in}$, and interchanging $v$ and $w$ (if necessary) so that $v$
appears before $w$, $W$ will be in one of these forms.
\[%
\begin{array}
[c]{ccc}%
x^{in}...v...w...v...w...x^{out}... & \text{ \ \ \ } & x^{in}...x^{out}%
...v...w...v...w...
\end{array}
\]
In the first case $\beta(x,v)+\beta(x,w)=1+1$, and in the second case
$\beta(x,v)+\beta(x,w)=0+0$. For the type (c) equations, if $vw,vx\in E(G)$
and $wx\notin E(G)$ then after cyclically permuting $W$ to begin with $v^{in}%
$, and interchanging $w$ and $x$ (if necessary) so that $w$ appears before
$x$, we may presume $W$ is in one of these forms:%
\[%
\begin{array}
[c]{ccc}%
v^{in}...w^{in}...x^{in}...v^{out}...x^{out}...w^{out}... & \text{ \ \ \ } &
v^{in}...w^{out}...x^{in}...v^{out}...x^{out}...w^{in}...\\
\text{ } & \text{ } & \text{ }\\
v^{in}...w^{in}...x^{out}...v^{out}...x^{in}...w^{out}... &  & v^{in}%
...w^{out}...x^{out}...v^{out}...x^{in}...w^{in}...
\end{array}
\]
Proceeding from left to right, the sum $\beta(v,w)+\beta(v,x)+\beta
(w,x)+\beta(x,w)$ is $1+1+0+1$ or $0+1+1+1$ for the words in the top row, and
$1+0+0+0$ or $0+0+1+0$ for the words in the bottom row.
\end{proof}

We say the Naji solution defined in the proof of Proposition \ref{circleN}
\emph{corresponds} to the orientation of $W$ used to define it. Notice that
cyclically permuting $W$ has no effect on the corresponding Naji solution.

\begin{definition}
For each vertex $v\in V(G)$, let $\delta(v)$ be given by $\delta(v)(v,w)=1$
$\forall w\neq v$, $\delta(v)(w,v)=1$ if $vw\in E(G)$, and $\delta(v)(x,y)=0$ otherwise.
\end{definition}

\begin{definition}
Let $\rho$ be given by $\rho(v,w)=1$ $\forall v\neq w\in V(G)$.
\end{definition}

If $\beta$ is the Naji solution of $\mathcal{I}(W)$ corresponding to an
orientation of a double occurrence word $W$, then reversing $W$ results in the
Naji solution $\beta+\rho$. Also, if $v\in V(\mathcal{I}(W))$ then
$\beta+\delta(v)$ is the Naji solution of $\mathcal{I}(W)$ corresponding to
the orientation obtained by interchanging the appearances of $v^{in}$ and
$v^{out}$ in $W$. Consequently, $\beta+\rho$ and $\beta+\delta(v)$ are both
Naji solutions of $\mathcal{I}(W)$. A similar assertion holds for all Naji graphs:

\begin{proposition}
\label{solution}Let $\beta$ be a Naji solution for $G$. Then $\beta+\rho$ is a
Naji solution and for each $v\in V(G)$, $\beta+\delta(v)$ is a Naji solution.
\end{proposition}

\begin{proof}
The fact that $\beta+\rho$ is a Naji solution follows from the fact that every
Naji equation has an even number of summands. For $\beta+\delta(v)$, verifying
the Naji equations is a little more delicate, but it turns out that each
equation has an even number of terms to which $\delta(v)$ makes a nonzero
contribution. For example, if we consider the vertex $x$ in a type (c) Naji
equation, $\delta(x)$ contributes a 1 to to the $\beta(v,x)$ and $\beta(w,x)$
terms, but does not contribute to the $\beta(v,w)$ and $\beta(x,w)$ terms.
\end{proof}

\begin{corollary}
\label{solutions}If $G$ is a Naji graph and $v\in V(G)$, then for every subset
$X\subseteq V(G)-v$, $G$ has a Naji solution $\beta$ such that $X=\{x\in
V(G)\mid\beta(x,v)=1\}$.
\end{corollary}

\begin{proof}
Begin with an arbitrary Naji solution $\beta$ and consider the Naji solution%
\[
\beta+\sum_{\substack{x\in X\\\beta(x,v)=0}}\delta(x)+\sum_{\substack{y\notin
X\cup\{v\}\\\beta(y,v)=1}}\delta(y)\text{.}%
\]

\end{proof}

Another corollary expresses an important insight: the space of all Naji
solutions provides more information than any individual Naji solution does.

\begin{corollary}
A Naji graph is determined up to isomorphism by its Naji solutions. However,
nonisomorphic Naji graphs of the same order may share some Naji solutions.
\end{corollary}

\begin{proof}
Suppose $G$ is a Naji graph, and $v_{0}\in V(G)$. We say two Naji solutions of
$G$ are \emph{related at} $v_{0}$ if $\beta(v_{0},y)\neq\beta^{\prime}%
(v_{0},y)$ whenever $v_{0}\neq y$. Proposition \ref{solution} tells us that
every Naji solution $\beta$ is related at $v_{0}$ to at least one other Naji
solution, as $\beta$ and $\beta+\delta(v_{0})$ are related at $v_{0}$.

Let $v\neq v_{0}\in V(G)$. If $vv_{0}\notin E(G)$, then for any Naji solution
$\beta$, $\beta$ and $\beta+\delta(v_{0})$ have $\beta(v,v_{0})=(\beta
+\delta(v_{0}))(v,v_{0})$. If $vv_{0}\in E(G)$ then consider any pair of Naji
solutions that are related at $v_{0}$, $\beta$ and $\beta^{\prime}$. The Naji
equations require $\beta(v,v_{0})\neq\beta(v_{0},v)\neq\beta^{\prime}%
(v_{0},v)\neq\beta^{\prime}(v,v_{0})$, so $\beta(v,v_{0})\neq\beta^{\prime
}(v,v_{0})$. We conclude that $vv_{0}\not \in E(G)$ if and only if $G$ has a
pair of Naji solutions that are related at $v_{0}$ and have $\beta
(v,v_{0})=\beta^{\prime}(v,v_{0})$. Consequently, the Naji solutions of $G$
determine the neighbors of $v_{0}$. Of course there is nothing special about
$v_{0}$, so the Naji solutions of $G$ determine all vertex-neighborhoods in
$G$.

There are many examples of the second sentence of the statement. For instance,
every Naji solution of the connected graph of order 2 is also a Naji solution
of the disconnected graph of order 2. In fact, it takes some patience to find
two small Naji graphs of the same order that do not share a Naji solution. For
the reader who might want to find such examples without our guidance, here is
a \textquotedblleft spoiler alert\textquotedblright: do not read the last
sentence of Section 10.
\end{proof}

\begin{proposition}
\label{indep}The set $\{\rho\}\cup\{\delta(v)\mid v\in V(G)\}$ is linearly
independent over $GF(2)$ unless $G$ is a complete graph, a star or a trivial
(edgeless) graph. If $\left\vert V(G)\right\vert =n>2$ then in each of these
exceptional cases the rank of $\{\rho\}\cup\{\delta(v)\mid v\in V(G)\}$ is $n$.
\end{proposition}

\begin{proof}
Suppose first that $\varnothing\neq S\subseteq V(G)$ and
\[
\sum_{s\in S}\delta(s)=0\text{.}%
\]
If $s\in S$ and $v\notin S$ then the $(s,v)$ coordinate of the sum is 1, as
only $\delta(s)$ has a nonzero $(s,v)$ coordinate. The sum is $0$, so we
conclude that $S=V(G)$. If $v\neq w$ then only $\delta(v)$ and $\delta(w)$ can
have nonzero $(v,w)$ coordinates, and both are nonzero only if $vw\in E(G)$;
hence $G=K_{n}$. In this case every subset $X\subseteq V(G)$ has
\[
\sum_{x\in X}\delta(x)\not =\rho\text{,}%
\]
for either $\left\vert X\right\vert \leq1$ and there is a coordinate $(v,w)$
that does not appear in the sum, or $\left\vert X\right\vert \geq2$ and there
is a coordinate $(v,w)$ that occurs precisely two times. Either way, the
$(v,w)$ coordinate of the sum is $0$. \ We conclude that the rank of
$\{\rho\}\cup\{\delta(v)\mid v\in V(G)\}$ is $n$.

Now, suppose that $S\subseteq V(G)$ and the equation%
\[
\sum_{s\in S}\delta(s)=\rho
\]
holds. If $S=V(G)$ then the equality requires $E(G)=\varnothing$. If $S\neq
V(G)$ then notice that for every pair of distinct vertices $v$ and $w$, there
must be a summand with a nonzero $(v,w)$ coordinate; it follows that at least
one of $v,w$ is an element of $S$. As this holds for every pair of distinct
vertices and $S\neq V(G)$, it must be that $\left\vert S\right\vert
=\left\vert V(G)\right\vert -1$. The one $v\notin S$ must be adjacent to every
$s\in S$, for if $vs\notin E(G)$ then no summand would have a nonzero $(v,s)$
coordinate.\ Also, no two vertices $s,s^{\prime}\in S$ can be neighbors; if
they were, then both $\delta(s)$ and $\delta(s^{\prime})$ would have nonzero
$(s,s^{\prime})$ coordinates, and the two summands would cancel. Consequently
$G$ is a star graph with the vertices in $S$ all of degree 1.
\end{proof}

\begin{proposition}
\label{consecutive}Suppose $v$ and $w$ are two vertices of a connected circle
graph $G=\mathcal{I}(W)$. Then $v$ and $w$ appear consecutively in $W$ if and
only if there is an orientation of $W$ for which the corresponding Naji
solution has $\beta(x,v)=\beta(x,w)$ $\forall x\notin\{v,w\}$.
\end{proposition}

\begin{proof}
If $W$ has an orientation in which $v^{out}$ and $w^{out}$ appear
consecutively, then the corresponding Naji solution has $\beta(x,v)=\beta
(x,w)$ $\forall x\notin\{v,w\}$.

For the converse, suppose $W$ can be oriented in such a way that the
corresponding Naji solution has $\beta(x,v)=\beta(x,w)$ $\forall
x\notin\{v,w\}$. Permute $W$ cyclically so that it is in the form
$Av^{out}Bw^{out}$; this permutation does not affect the associated Naji
solution. If $A$ or $B$ is empty, then $v^{out}$ and $w^{out}$ are
consecutive. Suppose instead that $A$ and $B$ are both nonempty; we derive
contradictions in all cases. Consider an arbitrary $x\notin\{v,w\}$. If
$x^{in}$ appears in $A$ and $x^{out}$ appears in $B$ then $\beta
(x,v)=0\neq\beta(x,w)$, a contradiction. Also, if $x^{out}$ appears in $A$ and
$x^{in}$ appears in $B$ then $\beta(x,v)=1\neq\beta(x,w)$, another
contradiction. Consequently, for every $x\notin\{v,w\}$, both $x^{in}$ and
$x^{out}$ must appear in the same one of $A,B$. If neither $v^{in}$ nor
$w^{in}$ appears in $A$, it follows that no vertex that appears in $A$ is
interlaced with a vertex that does not appear in $A$; but then $G$ is not
connected, an impossibility. Similarly, if neither $v^{in}$ nor $w^{in}$
appears in $B$ then $G$ is not connected. Consequently one of $v^{in},w^{in}$
appears in $A$ and the other appears in $B$. If $v^{in}$ appears in $A$ then
the subwords $Av^{out}$ and $Bw^{out}$ are separate double occurrence words;
and if $w^{in}$ appears in $A$ then the subwords $v^{out}B$ and $w^{out}A$ are
separate double occurrence words. Either way, the connectedness of $G$ is contradicted.
\end{proof}

\section{Prime graphs and splits}

Cunningham's split decomposition \cite{Cu} is of fundamental importance in
analyzing circle graphs.

\begin{definition}
Let $G$ be a simple graph. A \emph{split} $(X$, $Y)$ of $G$ is given by a
partition $V(G)=X\cup Y$ with $\left\vert X\right\vert $, $\left\vert
Y\right\vert \geq2$ and subsets $X^{\prime}\subseteq X$ and $Y^{\prime
}\subseteq Y$ such that the set of edges of $G$ connecting $X$ to $Y$ is
$\{xy\mid x\in X^{\prime}$ and $y\in Y^{\prime}\}$.
\end{definition}

Connected graphs of order 1, 2 or 3 have no splits, for the trivial reason
that $2+2>3$. On the other hand, it is easy to see that every graph of order 4
has a split. For $n\geq5$ a graph with no split is called \emph{prime}.

If $G$ has a split $(X,Y)$ then Cunningham called $G$ the \emph{composition}
of two smaller graphs, $G_{X}$ and $G_{Y}$. $G_{X}$ is obtained from the
induced subgraph $G[X]$ by attaching a new vertex $y_{0}$, with $N(y_{0}%
)=X^{\prime}$. $G_{Y}$ is obtained in the same way from $G[Y]$, except the new
vertex is denoted $x_{0}$. The new vertices $x_{0}$ and $y_{0}$ are called
\emph{markers.} Cunningham actually used only one marker but it is convenient
to use two for the simple reason that $G_{X}$ and $G_{Y}$ are then disjoint
graphs of orders strictly less than the order of $G$, so inductive arguments
can be set up in a natural way. Moreover, if $X^{\prime}\neq\varnothing\neq
Y^{\prime}$ then $G_{X}$ and $G_{Y}$ are both isomorphic to full subgraphs of
$G$: $G_{X}\cong G[X\cup\{y_{0}\}]$ for any $y_{0}\in Y^{\prime}$ and
$G_{Y}\cong G[Y\cup\{x_{0}\}]$ for any $x_{0}\in X^{\prime}$.

If a connected graph has a split then Cunningham showed that the graph can be
decomposed in an essentially unique way using compositions of smaller graphs.
This unique decomposition is both elegant and useful, but we do not discuss it
in detail because uniqueness of the split decomposition is not crucial here.

The following simple proposition of Bouchet \cite{Bec} allows us to focus our
attention on prime circle graphs.

\begin{proposition}
\label{circles}\cite{Bec} If $G$ has a split $(X,Y)$, then $G$ is a circle
graph if and only if $G_{X}$ and $G_{Y}$ are both circle graphs.
\end{proposition}

\begin{proof}
If $G_{X}=\mathcal{I}(W_{1}y_{0}W_{3}y_{0})$ and $G_{Y}=\mathcal{I}(W_{2}%
x_{0}W_{4}x_{0})$ then $G=\mathcal{I}(W_{1}W_{2}W_{3}W_{4})$.

Suppose conversely that $G$ is a circle graph and $G=\mathcal{I}(W)$. After a
cyclic permutation, we may presume that $W$ begins with an element of $X$, and
ends with an element of $Y$. Then there is a unique way to write $W$ as
$W_{1}W_{2}...W_{2m}$ so that every $W_{i}$ with $i$ odd is nonempty and
contains only letters from $X$, while every $W_{i}$ with $i$ even is nonempty
and contains only letters from $Y$. If no edge of $G$ connects $X$ to $Y$ then
$G_{X}=\mathcal{I}(y_{0}y_{0}W_{1}W_{3}...W_{2m-1})$ and $G_{Y}=\mathcal{I}%
(x_{0}x_{0}W_{2}W_{4}...W_{2m})$. Otherwise, let $xy$ be an edge of $G$ with
$x\in X$ and $y\in Y$. After cyclic permutation we may presume that $x$
appears in $W_{1}$ and $W_{2i-1}$, and $y$ appears in $W_{2j}$ and $W_{2k}$,
with $i>1$ and $j<k$. The fact that $xy$ is an edge implies that $1\leq
j<i\leq k$. Then $G_{X}=\mathcal{I}(W_{1}...W_{2j-1}y_{0}W_{2j+1}%
...W_{2k-1}y_{0}W_{2k+1}$ $...W_{2m-1})$ and $G_{Y}=\mathcal{I}(x_{0}%
W_{2}...W_{2i-2}x_{0}W_{2i}...W_{2m}).$
\end{proof}

\section{Local complementation}

\begin{definition}
If $v$ is a vertex of a simple graph $G$ then the \emph{local complement
}$G^{v}$ is the graph obtained from $G$ by reversing the adjacency status of
every pair of neighbors of $v$. A graph that can be obtained from $G$ through
some sequence of local complementations is \emph{locally equivalent} to $G$.
\end{definition}

That is, $G^{v}$ includes the same edges $vw$ and $wx$ as $G$, so long as
$x\notin N(v)$; but if $y\neq z\in N(v)$ then $yz\in E(G^{v})$ if and only if
$yz\notin E(G)$.

Local complementation is important in the theory of circle graphs because the
following propositions indicate that inductive proofs involving prime circle
graphs can be set up using local complementation. The first two appeared in
Bouchet's discussion of his circle graph recognition algorithm \cite{Bec}.

\begin{proposition}
\cite{Bec} If $v\in V(G)$ and $(X,Y)$ is a partition of $V(G)$ then $(X,Y)$ is
a split of $G$ if and only if $(X,Y)$ is a split of $G^{v}$.
\end{proposition}

\begin{proposition}
\cite{Bec} If $G$ and $H$ are locally equivalent then $G$ is a circle graph if
and only if $H$ is a circle graph.
\end{proposition}

The next proposition appeared in Gasse's derivation of Naji's theorem \cite{G}.

\begin{proposition}
\label{localn}\cite{G} If $G$ and $H$ are locally equivalent then $G$ is a
Naji graph if and only if $H$ is a Naji graph.
\end{proposition}

\begin{proof}
Suppose $G$ is a Naji graph. According to Corollary \ref{solutions}, $G$ has a
Naji solution $\beta$ such that $\beta(x,v)=0$ if and only if $xv\in E(G)$. A
Naji solution for $G^{v}$ may then be defined by
\[
\beta^{v}(x,y)=\left\{
\begin{array}
[c]{c}%
\beta(x,y)+\beta(v,y)\text{ if }x\in N(v)\\
\beta(x,y)+\beta(v,x)\text{ if }x\notin N(v)\text{,}%
\end{array}
\ \ \right.
\]
with the understanding that $v\notin N(v)$ and $\beta(v,v)=0.$
\end{proof}

We can say a little more.

\begin{proposition}
\label{localn2}If $v\in V(G)$ then $\left\vert \mathcal{B}(G)\right\vert
=\left\vert \mathcal{B}(G^{v})\right\vert $.
\end{proposition}

\begin{proof}
If $G$ is not a Naji graph then Proposition \ref{localn} tells us that $G^{v}$
is not a Naji graph either.

Suppose $G$ is a Naji graph, and let $\mathcal{B}_{0}(G)$ be the set that
includes all the Naji solutions $\beta$ of $G$ with the property that
$\beta(x,v)=0$ if and only if $x\in N(v)$. Suppose $\beta$ is an arbitrary
Naji solution of $G$. Let $X=\{x\in N(v)\mid\beta(x,v)=1\}$ and $Y=\{y\notin
N(v)\cup\{v\}\mid\beta(y,v)=0\}$. Then%
\[
\beta+\sum_{x\in X}\delta_{G}(x)+\sum_{y\in Y}\delta_{G}(y)\in\mathcal{B}%
_{0}(G).
\]
Moreover, if $W\subseteq V(G-v)$ is any subset other than $X\cup Y$ then
\[
\beta+\sum_{w\in W}\delta_{G}(w)\not \in \mathcal{B}_{0}(G).
\]
As no two subsets $W\subseteq V(G-v)$ yield the same sum $\sum_{w\in W}%
\delta_{G}(w)$, we conclude that $\left\vert \mathcal{B}(G)\right\vert
=\left\vert \mathcal{B}_{0}(G)\right\vert \cdot2^{\left\vert V(G)\right\vert
-1}$. The same argument applies to $G^{v}$, so it suffices to prove that
$\left\vert \mathcal{B}_{0}(G^{v})\right\vert =\left\vert \mathcal{B}%
_{0}(G)\right\vert $.

Suppose $\beta\in\mathcal{B}_{0}(G)$, and let $\beta^{v}$ be the Naji solution
of $G^{v}$ defined in Proposition \ref{localn}. Notice that if $x\in N(v)$,
then $\beta^{v}(x,v)=\beta(x,v)=0$. Also, if $y\notin N(v)\cup\{v\}$ then
$\beta^{v}(y,v)=\beta(y,v)+\beta(v,y)=1+\beta(v,y)$. Consequently if we let
$Y_{\beta}=\{y\notin N(v)\cup\{v\}\mid\beta(v,y)=1\}$ then
\[
f(\beta)\equiv\beta^{v}+\sum_{y\in Y_{\beta}}\delta_{G^{v}}(y)\in
\mathcal{B}_{0}(G^{v}).
\]

We claim that $f:\mathcal{B}_{0}(G)\rightarrow\mathcal{B}_{0}(G^{v})$ is
injective, and consequently $\left\vert \mathcal{B}_{0}(G^{v})\right\vert
\geq\left\vert \mathcal{B}_{0}(G)\right\vert $. As $G=(G^{v})^{v}$, the claim
suffices to complete the proof.

Suppose $\beta,\beta^{\prime}\in\mathcal{B}_{0}(G)$. If $Y_{\beta}\neq
Y_{\beta^{\prime}}$, there is a $z\notin N(v)\cup\{v\}$ with $\beta
(v,z)\not =\beta^{\prime}(v,z)$. Then $\beta^{v}(v,z)=\beta(v,z)\neq
(\beta^{\prime})^{v}(v,z)=\beta^{\prime}(v,z)$. Moreover there is no $y\in
Y_{\beta}\cup Y_{\beta^{\prime}}$ such that $\delta_{G^{v}}(y)$ has a nonzero
$(v,z)$ coordinate, because $v\notin N(z)$ and $v\notin Y_{\beta}\cup
Y_{\beta^{\prime}}$. Consequently $f(\beta)(v,z)\neq f(\beta^{\prime})(v,z)$.

Now, suppose $\beta,\beta^{\prime}\in\mathcal{B}_{0}(G)$ and $f(\beta
)=f(\beta^{\prime})$. As we just saw, $f(\beta)=f(\beta^{\prime})$ requires
$Y_{\beta}=Y_{\beta^{\prime}}$, i.e., $\beta(v,y)=\beta^{\prime}(v,y)$
$\forall y\notin N(v)\cup\{v\}$. As $\beta(y,v)=0=\beta^{\prime}(y,v)$
$\forall y\in N(v)$, the Naji equations imply $\beta(v,y)=1=\beta^{\prime
}(v,y)$ $\forall y\in N(v)$. Consequently $\beta(v,y)=\beta^{\prime}(v,y)$
$\forall y\neq v$. The equalities $f(\beta)=f(\beta^{\prime})$ and $Y_{\beta
}=Y_{\beta^{\prime}}$ imply $\beta^{v}=(\beta^{\prime})^{v}$, and the
definition of $\beta^{v}$ in Proposition \ref{localn} makes it clear that the
equalities $\beta^{v}=(\beta^{\prime})^{v}$ and $\beta(v,y)=\beta^{\prime
}(v,y)$ $\forall y\neq v$ imply $\beta=\beta^{\prime}$.
\end{proof}

The next proposition is more difficult; it was first proved by Bouchet using
isotropic systems \cite{Bec, Bi2}.

\begin{proposition}
\label{deletion}\cite{Bec} If $G$ is prime and $\left\vert V(G)\right\vert >5$
then there is a locally equivalent graph $H$ with a vertex $v$ such that $H-v$
is prime.
\end{proposition}

A refined form of Proposition \ref{deletion} appears in Geelen's thesis
\cite{GeelenPhD}, which is freely available online. The reader who has not
already encountered Proposition \ref{deletion} is encouraged to read Geelen's
account, as the result is stronger and the proof does not require isotropic systems.

\begin{proposition}
\label{deletion2}\cite[Corollary 5.10]{GeelenPhD} If $G$ is prime and
$\left\vert V(G)\right\vert >5$ then either $G$ has a vertex $v$ such that
$G-v$ is prime, or $G$ has a degree-2 vertex $v$ such that $G^{v}-v$ is prime.
\end{proposition}

\section{Examples}

In this section we discuss some examples of the behavior of the Naji equations.

\subsection{Complete graphs}

The easiest Naji graphs to analyze are the complete graphs. If $n\geq2$ then
for each pair of vertices $v\neq w\in V(K_{n})$, either of $\beta(v,w)$,
$\beta(w,v)$ may be 1, and the other must be 0. Consequently $\left\vert
\mathcal{B}(K_{n})\right\vert =2^{n(n-1)/2}$.

Complete graphs exemplify a comment of Geelen mentioned in the introduction:
For $n\geq4$, $K_{n}$ has Naji solutions that do not come from Proposition
\ref{circleN}. To see why, notice that according to Definition \ref{interlace}%
, if $W$ is a double occurrence word with $\mathcal{I}(W)=K_{n}$ then the
vertices of $K_{n}$ can be ordered so that $W=v_{1}...v_{n}v_{1}...v_{n}$.
Suppose $W$ is oriented in such a way that the corresponding Naji solution has
$\beta(v_{1},v_{i})=1$ $\forall i>1$. Then the oriented version of $W$ must be
$v_{1}^{in}...v_{n}^{in}v_{1}^{out}...v_{n}^{out}$ or $v_{1}^{out}%
...v_{n}^{out}v_{1}^{in}...v_{n}^{in}$. It follows that $\beta(v_{n},v_{i})=0$
$\forall i<n$. We see that all Naji solutions of $K_{n}$ which arise from
Proposition \ref{circleN} have this property:

\begin{quote}
If there is a vertex $v$ such that $\beta(v,x)=1$ $\forall x\neq v$, then
there is also a vertex $w$ such that $\beta(w,x)=0$ $\forall x\neq w$.
\end{quote}

For $n\geq4$ there are many\ Naji solutions of $K_{n}$ that do not satisfy
this property. For instance, one such solution has $\beta(v_{i},v_{j})=1$
whenever $i<j$, except that $\beta(v_{n-1},v_{n})=0$ and $\beta(v_{n}%
,v_{n-1})=1$.

\subsection{Cycle graphs}

Another interesting class of Naji graphs includes the cycle graphs $C_{n}$,
with $n\geq4$. We index the vertices $v_{1},...,v_{n}$ in the usual way, so
that $N(v_{i})=\{v_{i-1},v_{i+1}\}$ for each $i$, with indices considered
modulo $n$. The Naji equations require $\beta(v_{i},v_{i+1})\neq\beta
(v_{i+1},v_{i})$ for each $i$, $\beta(v_{i},v_{j})=\beta(v_{i},v_{i+2})$
whenever $\left\vert j-i\right\vert \geq2$, and $\beta(v_{i-1},v_{i}%
)+\beta(v_{i+1},v_{i})+\beta(v_{i-1},v_{i+1})+\beta(v_{i+1},v_{i-1})=1$ for
each $i$.

It turns out that these equations are dependent. To verify the dependence it
is convenient to use type (a) and (b) equations to rewrite each type (c)
equation $\beta(v_{i-1},v_{i})+\beta(v_{i+1},v_{i})+\beta(v_{i-1}%
,v_{i+1})+\beta(v_{i+1},v_{i-1})=1$ in the equivalent form $0=\beta
(v_{i-1},v_{i})+\beta(v_{i},v_{i+1})+\beta(v_{i-1},v_{i+1})+\beta
(v_{i+1},v_{i+3})$. Then observe that
\begin{align*}
0  &  =\sum_{i=1}^{n-1}(\beta(v_{i-1},v_{i})+\beta(v_{i},v_{i+1}%
)+\beta(v_{i-1},v_{i+1})+\beta(v_{i+1},v_{i+3}))\\
&  =\beta(v_{n-1},v_{n})+\beta(v_{n},v_{1})+\sum_{i=1}^{n-2}2\beta
(v_{i-1},v_{i})+\beta(v_{n-1},v_{1})+\beta(v_{1},v_{3})\\
&  +\sum_{i=2}^{n-2}2\beta(v_{i},v_{i+2})+2\beta(v_{n},v_{2})\\
&  =\beta(v_{n-1},v_{n})+\beta(v_{n},v_{1})+\beta(v_{n-1},v_{1})+\beta
(v_{1},v_{3})\text{,}%
\end{align*}
so the last type (c) equation follows from the other equations.

Now, suppose we have a Naji solution of $C_{n}$. Let $\beta_{i}=\beta
(v_{i},v_{i+2})$ for each $i$, and let $\beta_{0}=\beta(v_{1},v_{2})$. Then
all the other $\beta$ values are determined by $\beta_{0},...,\beta_{n}$:
$\beta(v_{2},v_{1})=1+\beta_{0}$, $\beta(v_{2},v_{3})=1+\beta(v_{2}%
,v_{1})+\beta(v_{1},v_{3})+\beta(v_{3},v_{1})=\beta_{0}+\beta_{1}+\beta_{3}$,
$\beta(v_{3},v_{2})=1+\beta(v_{2},v_{3})$, $\beta(v_{3},v_{4})=1+\beta
(v_{3},v_{2})+\beta(v_{2},v_{4})+\beta(v_{4},v_{2})=\beta_{0}+\beta_{1}%
+\beta_{2}+\beta_{3}+\beta_{4}$, etc. These formulas use every Naji equation
except the last type (c) equation, with each equation used once, so we
conclude that for each choice of values of $\beta_{0},...,\beta_{n}$ there is
exactly one Naji solution of $C_{n}$. It follows that $\left\vert
\mathcal{B}(C_{n})\right\vert =2^{n+1}$.

\subsection{Cycle-pendant graphs}

For $n\geq5$ let $C_{n}^{+}$ denote the graph obtained from $C_{n-1}$ by
adjoining a vertex $v_{0}$ whose only neighbor is $v_{1}$. Given a Naji
solution $\beta$ of $C_{n-1}$, arbitrarily choose values for $\beta
(v_{0},v_{1})$ and $\beta(v_{0},v_{2})$. Then $C_{n}^{+}$ has Naji equations
that require $\beta(v_{1},v_{0})=1+\beta(v_{0},v_{1})$, $\beta(v_{0}%
,v_{i})=\beta(v_{0},v_{2})$ for $2\leq i\leq n-1$, $\beta(v_{i},v_{0}%
)=\beta(v_{i},v_{1})$ for $3\leq i\leq n-2$, $\beta(v_{2},v_{0})=1+\beta
(v_{0},v_{2})+\beta(v_{0},v_{1})+\beta(v_{2},v_{1})$, and
\begin{align*}
\beta(v_{n-1},v_{0})  &  =1+\beta(v_{0},v_{n-1})+\beta(v_{0},v_{1}%
)+\beta(v_{n-1},v_{1})\\
&  =1+\beta(v_{0},v_{2})+\beta(v_{0},v_{1})+\beta(v_{n-1},v_{1})\text{.}%
\end{align*}
These are all the Naji equations of $C_{n}^{+}$ that are not Naji equations of
$C_{n-1}$, so if we are given a Naji solution $\beta$ of $C_{n-1}$, we may
arbitrarily choose values for $\beta(v_{0},v_{1})$ and $\beta(v_{0},v_{2})$,
and then determine a Naji solution of $C_{n}^{+}$ from $\beta(v_{0},v_{1})$,
$\beta(v_{0},v_{2})$ and $\beta$.

Every Naji solution of $C_{n}^{+}$ restricts to a Naji solution of $C_{n-1}$,
of course, and then can be obtained from its restriction in the manner just
described. We conclude that $\left\vert \mathcal{B}(C_{n}^{+})\right\vert
=4\left\vert \mathcal{B}(C_{n-1})\right\vert =2^{n+2}$.

\subsection{Wheel graphs}

The wheel graph $W_{n}$ is obtained from $C_{n}$ by adjoining a single vertex,
adjacent to all the vertices of $C_{n}$. We use the same notation as in the
above discussion of $C_{n}$, with the new vertex denoted $v_{n+1}$. If
$n\geq5$ then for each $i\in\{1,...,n\}$, $W_{n}$ has Naji equations
\begin{align*}
\beta(v_{n+1},v_{i})+\beta(v_{n+1},v_{i+2})+\beta(v_{i},v_{i+2})+\beta
(v_{i+2},v_{i})  &  =1\\
\text{and }\beta(v_{n+1},v_{i})+\beta(v_{n+1},v_{i+3})+\beta(v_{i}%
,v_{i+3})+\beta(v_{i+3},v_{i})  &  =1\text{,}%
\end{align*}
so $\beta(v_{n+1},v_{i})+\beta(v_{n+1},v_{i+2})+\beta_{i}+\beta_{i+2}%
=1=\beta(v_{n+1},v_{i})+\beta(v_{n+1},v_{i+3})+\beta_{i}+\beta_{i+3}$. It
follows that $\beta(v_{n+1},v_{i+2})+\beta_{i+2}=\beta(v_{n+1},v_{i+3}%
)+\beta_{i+3}$ for every $i$; consequently the sum $\beta(v_{n+1},v_{j}%
)+\beta_{j}$ is the same for every $j$. This is impossible, though, as the
Naji equations of $W_{n}$ require%
\begin{align*}
&  \beta(v_{n+1},v_{i})+\beta_{i}+\beta(v_{n+1},v_{i+2})+\beta_{i+2}\\
&  =\beta(v_{n+1},v_{i})+\beta(v_{n+1},v_{i+2})+\beta(v_{i},v_{i+2}%
)+\beta(v_{i+2},v_{i})=1\text{.}%
\end{align*}
We conclude that for $n\geq5$, $W_{n}$ is not a Naji graph.

\section{Step 1 of the proof: uniqueness}

As noted in Proposition \ref{circleN}, every circle graph is a Naji graph; the
interesting part of Naji's theorem is the converse. According to Proposition
\ref{circles}, it suffices to prove the converse for prime Naji graphs. The
first step of our proof is the following uniqueness result.

\begin{theorem}
\label{step1}Let $G$ be a prime Naji graph, and let $\beta_{0}$ be any
particular Naji solution of $G$. Then every other Naji solution of $G$ is%
\[
\beta_{0}+\sum_{s\in S}\delta(s)\text{ or \ }\beta_{0}+\rho+\sum_{s\in
S}\delta(s)
\]
for some subset $S\subseteq V(G)$.
\end{theorem}

\begin{proof}
Suppose first that $G=C_{5}$. Proposition \ref{solution} tells us that every
sum%
\[
\beta_{0}+\sum_{s\in S}\delta(s)\text{ or \ }\beta_{0}+\rho+\sum_{s\in
S}\delta(s)
\]
is a Naji solution of $G$, and Proposition \ref{indep} tells us that the
dimension of the subspace spanned by $\{\rho\}\cup\{\delta(v)\mid v\in
V(C_{5})\}$ is 6. As noted in\ Section 5, the solution space for the Naji
equations of $C_{5}$ is also of dimension 6. Consequently the theorem holds
for $C_{5}$.

According to Bouchet \cite[Lemma 3.1]{Bec} every prime graph of order 5 is
locally equivalent to $C_{5}$, so Proposition \ref{localn2} tells us that the
theorem holds for all prime graphs of order 5.

We proceed using induction on $\left\vert V(G)\right\vert >5$. By Propositions
\ref{localn} and \ref{deletion}, without loss of generality we may replace $G$
with a locally equivalent graph so that there is a vertex $v\in V(G)$ such
that $G-v$ is prime. Let $\beta_{1}$ be some Naji solution for $G$. Then
$\beta_{0}$ and $\beta_{1}$ define Naji solutions for $G-v$ by restriction,
and the inductive hypothesis asserts that $\beta_{1}|(G-v)$ is
\[
(\beta_{0}+\sum_{s\in S}\delta(s))|(G-v)\text{ or \ }(\beta_{0}+\rho
+\sum_{s\in S}\delta(s))|(G-v)
\]
for some subset $S\subseteq V(G-v)$. Replacing $\beta_{1}$ with $\beta
_{1}+\sum_{s\in S}\delta(s)$ or $\beta_{1}+\rho+\sum_{s\in S}\delta(s)$, we
may presume that $\beta_{1}|(G-v)=\beta_{0}|(G-v)$. That is,
\begin{equation}
\beta_{1}(x,y)=\beta_{0}(x,y)\text{ whenever }v\notin\{x,y\}. \tag{$\ast
$}\label{restrict}%
\end{equation}

The rest of the proof involves a detailed analysis of the structure of $G$. We
partition $V(G-v)$ into four sets.

\begin{itemize}
\item $A=\{a\in V(G-v)\mid\beta_{0}(a,v)=\beta_{1}(a,v)$ and $\beta
_{0}(v,a)=\beta_{1}(v,a)\}$

\item $B=\{b\in V(G-v)\mid\beta_{0}(b,v)=\beta_{1}(b,v)$ and $\beta
_{0}(v,b)\not =\beta_{1}(v,b)\}$

\item $C=\{c\in V(G-v)\mid\beta_{0}(c,v)\not =\beta_{1}(c,v)$ and $\beta
_{0}(v,c)=\beta_{1}(v,c)\}$

\item $D=\{d\in V(G-v)\mid\beta_{0}(d,v)\not =\beta_{1}(d,v)$ and $\beta
_{0}(v,d)\not =\beta_{1}(v,d)\}$
\end{itemize}

Claim 1. Both $B\cap N(v)=\varnothing$ and $C\cap N(v)=\varnothing$.

proof: The Naji equations require $\beta_{i}(v,x)\neq\beta_{i}(x,v)$ for
$i\in\{0,1\}$ when $vx$ is an edge, and both inequalities cannot hold if $x\in
B\cup C$. {\ \rule{0.5em}{0.5em}}

Claim 2. Either $A\cap N(v)=\varnothing$ or $C=\varnothing$.

proof: Suppose $a\in A\cap N(v)$ and $c\in C$. If $ac\notin E(G)$ then the
Naji equations require $\beta_{i}(c,a)=\beta_{i}(c,v)$ for $i=0$ and $1$. This
is not possible, as $\beta_{0}(c,v)\not =\beta_{1}(c,v)$ by the definition of
$C$ and $\beta_{0}(c,a)=\beta_{1}(c,a)$ by (\ref{restrict}). Hence $ac\in
E(G)$. Then the Naji equations require%
\[
\beta_{i}(a,c)+\beta_{i}(a,v)+\beta_{i}(v,c)+\beta_{i}(c,v)=1
\]
for $i\in\{0,1\}$. Both equations cannot be true as $\beta_{0}(c,v)\not =%
\beta_{1}(c,v)$ and the other terms are all equal. {\ \rule{0.5em}{0.5em}}

Claim 3. Either $D\cap N(v)=\varnothing$ or $D-N(v)=\varnothing$.

proof: Suppose $x\in D\cap N(v)$ and $y\in D-N(v)$. If $xy\in E(G)$ then as
$vy\notin E(G)$, the Naji equations require%
\[
\beta_{i}(x,y)+\beta_{i}(x,v)+\beta_{i}(v,y)+\beta_{i}(y,v)=1
\]
for $i\in\{0,1\}$. Both equations cannot be true as $\beta_{0}(x,y)=\beta
_{1}(x,y)$ by (\ref{restrict}), and the other terms are all unequal. Hence
$xy\notin E(G)$, so $\beta_{i}(y,x)=\beta_{i}(y,v)$ for $i\in\{0,1\}$. Both
equations cannot be true as $\beta_{0}(y,x)=\beta_{1}(y,x)$ by (\ref{restrict}%
) but the definition of $D$ requires $\beta_{0}(y,v)\not =\beta_{1}(y,v)$.
{\ \rule{0.5em}{0.5em}}

Claim 4. No edge connects $(A\cup C)-N(v)$ to $(B\cup D)-N(v)$.

proof: Suppose $x\in(A\cup C)-N(v)$ is adjacent to $y\in(B\cup D)-N(v)$. Then
the Naji equations require $\beta_{i}(v,x)=\beta_{i}(v,y)$ for $i\in\{0,1\}$,
but both equations cannot be true as $\beta_{0}(v,x)=\beta_{1}(v,x)$ and
$\beta_{0}(v,y)\not =\beta_{1}(v,y)$. {\ \rule{0.5em}{0.5em}}

Claim 5. Suppose $b\in B$ and $x\notin B$ are neighbors. Then $x\in D$.

proof: If $x\in A\cap N(v)$ then as $b\notin N(v)$, the Naji equations require%
\[
\beta_{i}(v,x)+\beta_{i}(b,x)+\beta_{i}(b,v)+\beta_{i}(v,b)=1
\]
for $i\in\{0,1\}$. This is impossible as $\beta_{0}(v,b)\neq\beta_{1}(v,b)$
and the other terms of the two equations are all equal. Claim 4 now implies
that $x\in D$. {\ \rule{0.5em}{0.5em}}

Claim 6. If $a\in A\cap N(v)$ then $a$ is adjacent to every element of $D$,
and the other neighbors of $a$ all lie in $A$.

proof: Suppose $a\in A\cap N(v)$ is not adjacent to $d\in D$. If $d\notin
N(v)$ then the Naji equations require $\beta_{i}(d,a)=\beta_{i}(d,v)$ for
$i\in\{0,1\}$, an impossibility as $\beta_{0}(d,a)=\beta_{1}(d,a)$ by
(\ref{restrict}) and $\beta_{0}(d,v)\not =\beta_{1}(d,v)$ by the definition of
$D$. If $d\in N(v)$ then the Naji equations require%
\[
\beta_{i}(v,d)+\beta_{i}(v,a)+\beta_{i}(a,d)+\beta_{i}(d,a)=1
\]
for $i\in\{0,1\}$, an impossibility as $\beta_{0}(v,d)\neq\beta_{1}(v,d)$ and
the other terms of the two equations are all equal.

For the second assertion, observe that claim 2 tells us that $C=\varnothing$
and claim 5 tells us that no $b\in B$ is a neighbor of $a$.
{\ \rule{0.5em}{0.5em}}

Claim 7. If $c\in C$ then $c$ is adjacent to every element of $D\cap N(v)$,
and the other neighbors of $c$ all lie in $(A\cup C)-N(v)$.

proof: If $x$ is a neighbor of $c$ then claim 2 implies that $x\notin A\cap
N(v)$, and claim 4 implies that $x\notin(B\cup D)-N(v)$. Hence $x\in((A\cup
C)-N(v))\cup(D\cap N(v))$. If $d\in D\cap N(v)$ and $cd\notin E(G)$ then the
Naji equations require $\beta_{i}(c,d)=\beta_{i}(c,v)$ for $i\in\{0,1\}$; but
this is impossible as $\beta_{0}(c,d)=\beta_{1}(c,d)$ by (\ref{restrict}) and
$\beta_{0}(c,v)\not =\beta_{1}(c,v)$ by the definition of $C$.
{\ \rule{0.5em}{0.5em}}

Claim 8. If $a\in A-N(v)$ then the neighbors of $a$ all lie in $A\cup C$.

proof: Claim 4 implies that the neighbors of $a$ all lie in $A\cup C\cup(D\cap
N(v))$, so it suffices to verify that no neighbor of $a$ lies in $D\cap N(v)$.
Suppose instead that $d\in D\cap N(v)$ is a neighbor of $a$. The Naji
equations require%
\[
\beta_{i}(v,d)+\beta_{i}(a,d)+\beta_{i}(a,v)+\beta_{i}(v,a)=1
\]
for $i\in\{0,1\}$. This is impossible as $\beta_{0}(v,d)\neq\beta_{1}(v,d)$
and the other terms of the two equations are all equal. {\ \rule{0.5em}{0.5em}%
}

Claims 2 and 3 yield four cases.

Case 1. $C\not =\varnothing$ and $A\cap N(v)=\varnothing=D\cap N(v)$. In this
case claim 1 tells us that $N(v)=\varnothing$, an impossibility as a prime
graph cannot have an isolated vertex.

Case 2. $C\not =\varnothing$ and $A\cap N(v)=\varnothing=D-N(v)$. In this case
claim 1 tells us that $D=N(v)$, claim 7 tells us that the elements of $C$ are
all adjacent to all the elements of $D$ and the other neighbors of elements of
$C$ all lie in $A\cup C$, claim 8 tells us that the neighbors of elements of
$A$ all lie in $A\cup C$, and claim 5 tells us that the neighbors of elements
of $B$ all lie in $B\cup D$. As $(A\cup C,B\cup D\cup\{v\})$ cannot be a split
of $G$, either $\left\vert A\cup C\right\vert \leq1$ or $B\cup D=\varnothing$.
The latter is impossible as it would leave $v$ isolated. As $C\not =%
\varnothing$, $\left\vert A\cup C\right\vert \leq1$ implies $A=\varnothing$
and $\left\vert C\right\vert =1$. Then the lone $c\in C$ has $N(c)=D=N(v)$;
but this is impossible as $(B\cup D,\{c,v\})$ would be a split of $G$.

Case 3. $C=\varnothing$ and $D-N(v)=\varnothing$. In this case all elements of
$A\cap N(v)$ are adjacent to all elements of $D$ (claim 6), the neighbors of
elements of $B$ all lie in $B\cup D$ (claim 5), and the neighbors of elements
of $A-N(v)$ all lie in $A$ (claim 4). As neither $(A\cup\{v\},B\cup D)$ nor
$(A,B\cup D\cup\{v\})$ is a split of $G$, either $A=\varnothing$ or $B\cup
D=\varnothing$. If $A=\varnothing=C$ then $D=N(v)$ and $B=V(G-v)-N(v)$; hence
$\beta_{1}=\beta_{0}+\delta(v)$. If $B\cup D=\varnothing=C$ then $\beta
_{1}=\beta_{0}$.

Case 4. $C=\varnothing$ and $D\cap N(v)=\varnothing$. In this case
$N(v)\subseteq A$ (claim 1), the neighbors of elements of $A-N(v)$ all lie in
$A$ (claim 4), every element of $N(v)$ is adjacent to every element of $D$
(claim 6), and the neighbors of elements of $B$ all lie in $B\cup D$ (claim
5). As $(A\cup\{v\},B\cup D)$ cannot be a split of $G$, it follows that
$\left\vert A\right\vert =0$ or $\left\vert B\cup D\right\vert \leq1$. If
$\left\vert A\right\vert =0$ then $v$ is isolated, an impossibility in a prime
graph. If $\left\vert B\right\vert =1$ then the lone $b\in B$ is isolated,
another impossibility. If $B=C=\varnothing\neq D$ then there is a lone $d\in
D$, with $N(d)=N(v)$. But that too is impossible, as $(A,\{d,v\})$ would be a
split of $G$. Consequently $B=C=D=\varnothing$, so $\beta_{1}=\beta_{0}$.
\end{proof}

A corollary of Theorem \ref{step1} describes the relationship between the Naji
solutions of $G$ and those of $G-v$, in case both graphs are prime.

\begin{corollary}
\label{surject}Let $G$ be a prime Naji graph with $\left\vert V(G)\right\vert
\geq6$, and suppose $G-v$ is also prime. Then restriction defines a 2-to-1
surjection%
\[
\{\text{Naji solutions of }G\}\twoheadrightarrow\{\text{Naji solutions of
}G-v\}\text{.}%
\]

\end{corollary}

\begin{proof}
If $\beta$ is any Naji solution of $G$ then certainly the restriction
$\beta|(G-v)$ is a Naji solution of $G-v$. As the $\rho$ and $\delta(x)$
vectors of $G$ restrict to those of $G-v$ (with the exception that $\delta(v)$
restricts to $0$) Theorem \ref{step1} guarantees that restriction defines a
surjection. To verify that the surjection is 2-to-1, i.e., every Naji solution
of $G-v$ corresponds to precisely two Naji solutions of $G$, note first that
every Naji solution of $G-v$ corresponds to at least two different Naji
solutions of $G$; there must be one, as restriction is surjective, and then
there is another obtained by adding $\delta(v)$. Then note that Proposition
\ref{indep} implies that there are twice as many Naji solutions for $G$ as
there are for $G-v$.
\end{proof}

Another corollary is the following result of Bouchet \cite{Bec}.

\begin{corollary}
\label{oneword}\cite{Bec} Let $G$ be a prime circle graph. Then there is only
one double occurrence word $W$ with $G=\mathcal{I}(W)$, up to cyclic
permutation and reversal.
\end{corollary}

\begin{proof}
Let $W$ and $W^{\prime}$ be double occurrence words with $\mathcal{I}%
(W)=\mathcal{I}(W^{\prime})=G$. Theorem \ref{step1} tells us that $W$ and
$W^{\prime}$ have orientations whose corresponding Naji solutions are the
same. Choose any $v\in V(G)$, and cyclically permute $W$ and $W^{\prime}$ so
they both begin with $v^{out}$. As $v$ is not isolated in $G$, the second
letter in $W$ is some $w\neq v$. Adding $\delta(w)$ to both Naji solutions if
necessary, we may presume that this second letter of $W$ is $w^{out}$. Then
$\beta(x,v)=\beta(x,w)$ $\forall x\notin\{v,w\}$; as $W$ and $W^{\prime}$
provide the same Naji solution, Proposition \ref{consecutive} guarantees that
$v^{out}$ and $w^{out}$ are consecutive in $W^{\prime}$ too. It is possible
that $w^{out}$ is the last letter in $W^{\prime}$, rather than the second; if
so, reverse and cyclically permute $W^{\prime}$ so that both $W$ and
$W^{\prime}$ are in the form $v^{out}w^{out}...$ Repeat this process to verify
that $W$ and $W^{\prime}$ must have the same third letter, then the same
fourth letter, and so on.
\end{proof}

The appearance of Bouchet's Corollary \ref{oneword} here is no coincidence.
Our proof of Naji's theorem follows the outline of the argument given by
Bouchet in justifying his circle graph recognition algorithm \cite{Bec}.
However, the second part of the proof is considerably more difficult for us.
The second part of Bouchet's algorithm used simple \textquotedblleft brute
force\textquotedblright\ (as he described it on p. 253 of \cite{Bec}) to check
all possible double occurrence words for a prime graph $G$, knowing that the
essentially unique double occurrence word for a prime circle graph $G$ must
arise from the essentially unique double occurrence word for a prime $G-v$.
Our job will be more difficult, as we must prove that a prime Naji graph
arises from a double occurrence word. Before completing this job in\ Section
8, we take a moment to discuss the converse of Theorem \ref{step1}.

\section{Counting Naji solutions}

An anonymous reader mentioned that for $n\geq5$ Corollary \ref{oneword}\ has a
valid converse, which was discussed by Gabor, Supowit and Hsu \cite{GSH}; a
complete proof of the converse was given by\ Courcelle \cite{C}. In this
section we show that Theorem \ref{step1} also has a valid converse for
$n\geq5$: a Naji graph is prime if and only if it has precisely $2^{n+1}$ Naji
solutions. This result is not part of our proof of Naji's theorem, so the
reader who is primarily interested in that proof may proceed to Section 8.

\begin{proposition}
\label{components}Let $G_{1}$ and $G_{2}$ be disjoint graphs of orders $n_{1}$
and $n_{2}$, respectively. Let $G=G_{1}\cup G_{2}$ be their union, with no
edge connecting $G_{1}$ to $G_{2}$. Then $G$ is a Naji graph if and only if
$G_{1}$ and $G_{2}$ are both Naji graphs. Moreover,%
\[
\left\vert \mathcal{B}(G)\right\vert \geq2^{n}\left\vert \mathcal{B}%
(G_{1})\right\vert \left\vert \mathcal{B}(G_{2})\right\vert \text{.}%
\]

\end{proposition}

\begin{proof}
If either of $G_{1},G_{2}$ is a non-Naji graph then of course $G$ is not a
Naji graph. Suppose instead that $G_{1}$ and $G_{2}$ are both Naji graphs.

If $1=n_{1}=n_{2}$ then $G$ has no Naji equation, so each of $\beta
(v_{1},v_{2})$, $\beta(v_{2},v_{1})$ may be $0$ or $1$; here $V(G_{i}%
)=\{v_{i}\}$. Consequently $\left\vert \mathcal{B}(G)\right\vert =4$ in this case.

Suppose $1=n_{1}<n_{2}$ and $\beta$ is a Naji solution of $G_{2}$. Let
$\{v_{1}\}=V(G_{1})$. If $v\in V(G_{2})$ then no Naji equation of $G$ mentions
$\beta(v,v_{1})$, so an arbitrary value may be chosen for $\beta(v,v_{1})$.
The only Naji equations of $G$ that mention a value $\beta(v_{1},v)$ are
equations of type (b), and these equations are certainly satisfied if the
$\beta(v_{1},v)$ values are all the same. Consequently each Naji equation of
$G_{2}$ yields at least $2^{n}$ different Naji solutions of $G$.

Suppose now that $2\leq n_{1}\leq n_{2}$\ and for $i\in\{1,2\}$, $\beta_{i}$
is a Naji solution of $G_{i}$. For each $v\in V(G)$ let $\beta_{v}\in GF(2)$
be arbitrary. The only Naji equations of $G$ that involve vertices from both
$G_{1}$ and $G_{2}$ are those of type (b), and these equations are all
satisfied by defining%
\[
\beta(v,w)=\left\{
\begin{array}
[c]{l}%
\beta_{i}(v,w)\text{, if }i\in\{1,2\}\text{ and }v,w\in V(G_{i})\\
\beta_{v}\text{, if }i\in\{1,2\}\text{, }v\in V(G_{i})\text{ and }w\notin
V(G_{i})
\end{array}
\right.  \text{.}%
\]
The values $\beta_{v}$ are arbitrary, so $\left\vert \mathcal{B}(G)\right\vert
\geq2^{n}\left\vert \mathcal{B}(G_{1})\right\vert \left\vert \mathcal{B}%
(G_{2})\right\vert $.
\end{proof}

\begin{proposition}
\label{splitsolutions}If a connected graph $G$ has a split $(X,Y)$, then $G$
is a Naji graph if and only if $G_{X}$ and $G_{Y}$ are both Naji graphs.
Moreover,
\[
\left\vert \mathcal{B}(G)\right\vert \geq\left\vert \mathcal{B}(G_{X}%
)\right\vert \left\vert \mathcal{B}(G_{Y})\right\vert /2.
\]

\end{proposition}

\begin{proof}
$G_{X}$ is isomorphic to a full subgraph of $G$, namely $G[X\cup\{y_{0}\}]$
for any $y_{0}\in Y^{\prime}$. Similarly, $G_{Y}$ $\cong G[Y\cup\{x_{0}\}]$
for any $x_{0}\in X^{\prime}$. We conclude that if either $G_{X}$ or $G_{Y}$
is not a Naji graph, then $G$ cannot be a Naji graph either. \ In this case
the inequality of the statement is satisfied because both sides are $0$.

If $G_{X}$ and $G_{Y}$ are both Naji graphs then by Corollary \ref{solutions},
$G_{X}$ has a Naji solution $\beta_{X}$ such that $\beta_{X}(x,y_{0})=1$
$\forall x\in X$, and $G_{Y}$ has a Naji solution $\beta_{Y}$ such that
$\beta_{Y}(y,x_{0})=0$ if and only if $y\in Y^{\prime}$. Given such $\beta
_{X}$ and $\beta_{Y}$, define $\beta$ as follows:%
\[
\beta(v,w)=\left\{
\begin{array}
[c]{l}%
\beta_{X}(v,w)\text{, if }v,w\in X\\
\beta_{Y}(v,w)\text{, if }v,w\in Y\\
\beta_{X}(y_{0},w)\text{, if }v\in Y^{\prime}\text{ and }w\in X\text{ }\\
1\text{, if }v\in Y-Y^{\prime}\text{ and }w\in X\\
\beta_{Y}(x_{0},w)\text{, if }v\in X^{\prime}\text{ and }w\in Y\\
1\text{, if }v\in X-X^{\prime}\text{ and }w\in Y
\end{array}
\right.  \text{.}%
\]

We claim that $\beta$ is a Naji solution of $G$. As $\beta_{X}$ and $\beta
_{Y}$ satisfy all Naji equations involving only vertices from $X$ or only
vertices from $Y$, to verify the claim it suffices to consider each equation
that involves at least one vertex from $X$ and at least one vertex from $Y$.
If $v\in X^{\prime}$ and $w\in Y^{\prime}$ then $\beta(v,w)=\beta_{Y}%
(x_{0},w)=1+\beta_{Y}(w,x_{0})=1+0=1$ and $\beta(w,v)=\beta_{X}(y_{0}%
,v)=1+\beta_{X}(v,y_{0})=1+1=0$, so the type (a) Naji equation $\beta
(v,w)=\beta(w,v)+1$ is satisfied. Type (b) Naji equations $\beta
(v,x)=\beta(v,w)$ arise in several situations: $v\in X-X^{\prime}$ and $w,x\in
Y$ (in which case both $\beta$ values are $1$); $v\in X-X^{\prime}$, $w\in
X^{\prime}$ and $x\in Y^{\prime}$ (in which case $\beta(v,x)=1$ and
$\beta(v,w)=\beta_{X}(v,w)=\beta_{X}(v,y_{0})=1$); $v\in X^{\prime}$ and
$w,x\in Y-Y^{\prime}$ (in which case $\beta(v,w)=\beta_{Y}(x_{0},w)$ and
$\beta(v,x)=\beta_{Y}(x_{0},x)$, and these two are equal according to a type
(b) Naji equation of $G_{Y}$); and similar situations in which $X$ and $Y$
have been interchanged. If $v\in X^{\prime}$, $w\in Y^{\prime}$, $x\in Y$ and
$wx\notin E(G)$ then the type (c) Naji equation $\beta(v,w)+\beta
(v,x)+\beta(w,x)+\beta(x,w)=1$ is satisfied in $G$ because $\beta_{Y}%
(x_{0},w)+\beta_{Y}(x_{0},x)+\beta_{Y}(w,x)+\beta_{Y}(x,w)=1$ is satisfied in
$G_{Y}$. If $v\in X^{\prime}$, $w\in Y^{\prime}$ and $x\in N(v)\cap
(X-X^{\prime})$ then $\beta(v,w)=1+\beta(w,v)=1+\beta_{X}(y_{0},v)=\beta
_{X}(v,y_{0})$, so the corresponding type (c) Naji equation of $G$ is
satisfied because%
\begin{align*}
&  \beta(v,w)+\beta(v,x)+\beta(w,x)+\beta(x,w)\\
&  =\beta_{X}(v,y_{0})+\beta_{X}(v,x)+\beta_{X}(y_{0},x)+1\\
&  =\beta_{X}(v,y_{0})+\beta_{X}(v,x)+\beta_{X}(y_{0},x)+\beta_{X}(x,y_{0})
\end{align*}
and the last line equals 1 by a Naji equation of $G_{X}$. Similar situations
in which $X$ and $Y$ are reversed are verified in similar ways.

The claim verifies the assertion that $G$ is a Naji graph. To verify the
inequality of the statement, define a mapping $f:\mathcal{B}(G_{X}%
)\times\mathcal{B}(G_{Y})\rightarrow\mathcal{B}(G)$ as follows. If $\beta_{1}$
is a Naji solution of $G_{X}$ and $\beta_{2}$ is a Naji solution of $G_{Y}$,
let $S_{1}=\{x\in X\mid\beta_{1}(x,y_{0})=0\}$ and $S_{2}=\{y\in Y^{\prime
}\mid\beta_{2}(y,x_{0})=1\}\cup\{y\in Y-Y^{\prime}\mid\beta_{2}(y,x_{0})=0\}$.
Then
\[
\beta_{X}=\beta_{1}+\sum_{x\in S_{1}}\delta_{G_{X}}(x)\text{ and }\beta
_{Y}=\beta_{2}+\sum_{y\notin S_{2}}\delta_{G_{Y}}(y)
\]
satisfy the requirements of the preceding paragraph. If $\beta$ is the Naji
solution of $G$ discussed there, then let
\[
f(\beta_{1},\beta_{2})=\beta+\sum_{x\in S_{1}}\delta_{G}(x)+\sum_{y\notin
S_{2}}\delta_{G}(y)\text{.}%
\]

Notice that we can \emph{almost} determine $S_{1}$ and $S_{2}$ from
$f(\beta_{1},\beta_{2})$. If $x\in X-X^{\prime}$ then $x\in S_{1}$ if and only
if $f(\beta_{1},\beta_{2})(x,y_{0})\neq\beta(x,y_{0})=1$. If $y\in
Y-Y^{\prime}$ then $y\in S_{2}$ if and only if $f(\beta_{1},\beta_{2}%
)(y,x_{0})\neq\beta(y,x_{0})=1$. Let us assume for the moment that $y_{0}\in
S_{2}$. Then if $x\in X^{\prime}$, $x\in S_{1}$ if and only if $f(\beta
_{1},\beta_{2})(x,y_{0})=\beta(x,y_{0})=1$. In particular, the preceding
sentence determines whether $x_{0}\in S_{1}$. Then for $y\in Y^{\prime}$,
$y\in S_{2}$ if and only if either $f(\beta_{1},\beta_{2})(y,x_{0}%
)=\beta(y,x_{0})=0$ and $x_{0}\in S_{1}$, or $f(\beta_{1},\beta_{2}%
)(y,x_{0})\neq\beta(y,x_{0})=0$ and $x_{0}\not \in S_{1}$. If we assume
$y_{0}\not \in S_{2}$ then we can determine $S_{1}$ and $S_{2}$ from
$f(\beta_{1},\beta_{2})$ in a similar way. Of course once we determine $S_{1}$
and $S_{2}$, we can determine $\beta_{1}$ and $\beta_{2}$ from $f(\beta
_{1},\beta_{2})$. We conclude that for each $\beta\in\mathcal{B}(G)$, there
are at most two distinct pairs $(\beta_{1},\beta_{2})\in\mathcal{B}%
(G_{X})\times\mathcal{B}(G_{Y})$ with $f(\beta_{1},\beta_{2})=\beta$; one pair
results from the assumption that $y_{0}\in S_{2}$, and the other pair results
from the assumption that $y_{0}\not \in S_{2}$. The inequality of the
statement follows.
\end{proof}

Before stating the main result of this section, we count Naji solutions for
graphs of orders $n\in\{1,2,3,4,5\}$.

If $n=1$ then $G$ has one vacuous Naji solution. Up to isomorphism, there are
two graphs with $n=2$. The disconnected graph has no Naji equations, so there
are 4 different Naji solutions. The connected graph has only two Naji
solutions, because of the requirement that $\beta(v,w)\neq\beta(w,v)$. Notice
that for $n\leq2$, $\left\vert \mathcal{B}(G)\right\vert \geq2^{n-1}$.

Up to isomorphism, there are four graphs with $n=3$. One graph has no edge; it
has $2^{6}$ Naji solutions. One graph has precisely one edge; it has $2^{4}$
Naji solutions. The two connected graphs are locally equivalent, so
Proposition \ref{localn2} tells us that they have the same number of Naji
solutions. One of the two is $K_{3}$, which has $2^{3}$ Naji solutions, as
noted in Section 5. Notice that for $n=3$, $\left\vert \mathcal{B}%
(G)\right\vert \geq2^{n}$.

Up to isomorphism, there are seven graphs with $n=4$. One graph has no edge,
and $2^{12}$ Naji solutions. One graph has precisely one edge, and $2^{9}$
Naji solutions. There are two graphs that have two edges. One of the two has
an isolated vertex; it has $2^{3}\cdot2^{4}=2^{7}$ Naji solutions. The other
has no isolated vertex; it has $2^{6}$ Naji solutions. There are two local
equivalence classes of connected 4-vertex graphs: one includes $K_{4}$, and
the other includes $C_{4}$. As observed in Section 5, they have $2^{6}$ and
$2^{5}$ Naji solutions, respectively. Notice that for $n=4$, $\left\vert
\mathcal{B}(G)\right\vert \geq2^{n+1}$.

For $n=5$, Proposition \ref{components} and the observations just given imply
that a disconnected graph has at least $2^{9}$ Naji solutions. If $G$ is
connected and has a split $(X,Y)$ we may presume $\left\vert X\right\vert =2$
and $\left\vert Y\right\vert =3$; then $G_{X}$ and $G_{Y}$ are of orders 3 and
4, respectively, so as noted above $\left\vert \mathcal{B}(G_{X})\right\vert
\geq2^{3}$ and $\left\vert \mathcal{B}(G_{Y})\right\vert \geq2^{5}$.
Proposition \ref{splitsolutions} tells us that $G$ has at least $2^{7}$
solutions. If $G$ is prime then Bouchet \cite[Lemma 3.1]{Bec} showed that $G$
is locally equivalent to $C_{5}$, so Proposition \ref{localn2} and the
discussion of Section 5 tell us that $\left\vert \mathcal{B}(G)\right\vert
=\left\vert \mathcal{B}(C_{5})\right\vert =2^{6}$. Notice that for $n=5$,
$\left\vert \mathcal{B}(G)\right\vert \geq2^{n+1}$; moreover $G$ realizes this
minimum if and only if $G$ is prime. The same pattern holds for larger values
of $n$:

\begin{corollary}
\label{moresplit} Let $G$ be a Naji graph with 5 or more vertices. If $G$ is
prime, then $\left\vert \mathcal{B}(G)\right\vert =2^{n+1}$. If $G$ has a
split, then $\left\vert \mathcal{B}(G)\right\vert \geq2^{n+2}$.
\end{corollary}

\begin{proof}
If $G$ is prime, the assertion follows from Proposition \ref{indep} and
Theorem \ref{step1}.

Suppose $G$ is not prime; we may assume that the corollary holds for graphs
smaller than $G$, with five or more vertices. Combining this inductive
hypothesis with the discussion above, we may assume that every Naji graph of
order $k\in\{3,...,n-1\}$ has at least $2^{k}$ Naji solutions, and if $k>3$
then the number of solutions is at least $2^{k+1}$.

If $G$ has an isolated vertex $v$ then Proposition \ref{components} tells us
that $\left\vert \mathcal{B}(G)\right\vert \geq2^{n}\left\vert \mathcal{B}%
(G-v)\right\vert $. The discussion of the preceding paragraph tells us that
$\left\vert \mathcal{B}(G-v)\right\vert $ $\geq2^{n}$, so $\left\vert
\mathcal{B}(G)\right\vert \geq2^{2n}>2^{n+2}$. If $G$ is disconnected but has
no isolated vertex then let $G_{1}$ be a smallest connected component of $G$,
and let $G_{2}$ be the complement of $G_{1}$ in $G$. Suppose $G_{1}$ and
$G_{2}$ are of orders $n_{1}$ and $n_{2}$ respectively. Then $n_{2}\geq3$, so
Proposition \ref{components} and the discussion of the preceding paragraph
tell us that
\[
\left\vert \mathcal{B}(G)\right\vert \geq2^{n}\cdot\left\vert \mathcal{B}%
(G_{1})\right\vert \cdot\left\vert \mathcal{B}(G_{2})\right\vert \geq
2^{n}\cdot2\cdot2^{n_{2}}>2^{n+2}.
\]
Suppose $G$ is connected and $G$ has a split $(X,Y)$ with $\left\vert
X\right\vert \leq$ $\left\vert Y\right\vert $. Then $\left\vert X\right\vert
\geq2$ and $\left\vert Y\right\vert \geq3$, so $G_{X}$ and $G_{Y}$ are of
orders $\left\vert X\right\vert +1\geq3$ and $\left\vert Y\right\vert +1\geq
4$, respectively. The discussion of the preceding paragraph tells us that
$\left\vert \mathcal{B}(G_{X})\right\vert \geq2^{\left\vert X\right\vert +1}$
and $\left\vert \mathcal{B}(G_{Y})\right\vert \geq2^{\left\vert Y\right\vert
+2}$, so Proposition \ref{splitsolutions} tells us that
\[
\left\vert \mathcal{B}(G)\right\vert \geq\left\vert \mathcal{B}(G_{X}%
)\right\vert \left\vert \mathcal{B}(G_{Y})\right\vert /2\geq2^{\left\vert
X\right\vert +1}\cdot2^{\left\vert Y\right\vert +2}/2=2^{\left\vert
X\right\vert +\left\vert Y\right\vert +2}=2^{n+2}\text{,}%
\]
as claimed.
\end{proof}

Note that the lower bound $2^{n+2}$ is attained by the cycle-pendant graphs
discussed in\ Section 5.

\section{Step 2 of the proof: building a word}

In this section we complete the proof of Naji's theorem. We begin with a
technical observation.

\begin{lemma}
\label{toggle}Let $G$ be a Naji graph with an edge $e=vw$. Suppose $G$ and
$G-e$ share a Naji solution $\beta$. Then $\beta(x,v)\neq\beta(x,w)$ $\forall
x\in(N(v)\Delta N(w))-\{v,w\}$, and $\beta(x,v)=\beta(x,w)$ $\forall
x\not \in N(v)\Delta N(w)$. (Here $\Delta$ denotes the symmetric difference.)
\end{lemma}

\begin{proof}
Suppose $x\in(N(v)\Delta N(w))-\{v,w\}$; renaming $v$ and $w$ if necessary, we
may presume that $vx\in E(G)$ and $wx\notin E(G)$. Then the Naji equations for
$G$ require
\[
\beta(x,v)+\beta(w,v)+\beta(x,w)+\beta(w,x)=1
\]
while the Naji equations for $G-e$ require $\beta(w,v)=\beta(w,x)$. If
$x\notin N(v)\cup N(w)$ then the Naji equations for $G$ require $\beta
(x,v)=\beta(x,w)$. If $x\in N(v)\cap N(w)$ then the Naji equations for $G-e$
require%
\[
\beta(x,v)+\beta(x,w)+\beta(v,w)+\beta(w,v)=1
\]
while the Naji equations for $G$ require $\beta(v,w)+\beta(w,v)=1$.
\end{proof}

\begin{corollary}
\label{toggle2}Let $G$ be a Naji graph with an edge $e=vw$. Suppose $G$ and
$G-e$ share a Naji solution $\beta$. Then $G$ and $G-e$ have Naji solutions
$\beta_{1}$ and $\beta_{2}$ (respectively) such that (a) the only difference
between $\beta_{1}$ and $\beta_{2}$ is that $\beta_{1}(v,w)=\beta_{2}(w,v)$
and $\beta_{1}(w,v)=\beta_{2}(v,w)$ and (b) $\beta_{1}(x,v)=\beta_{1}(x,w)$
$\forall x\notin\{v,w\}$.
\end{corollary}

\begin{proof}
Consider the Naji solutions $\beta_{1}=\beta+\delta_{G}(v)+\delta_{G}(w)$ and
$\beta_{2}=\beta+\delta_{G-e}(v)+\delta_{G-e}(w)$.
\end{proof}

\begin{theorem}
\label{step2}Let $G$ be a prime Naji graph with $\left\vert V(G)\right\vert
\geq5$. Then there is a double occurrence word $W$ with $\mathcal{I}(W)=G$.
\end{theorem}

\begin{proof}
In Section 7 we verified that all simple graphs of order $\leq5$ are Naji
graphs. All of them are circle graphs, too; in particular, $C_{5}%
=\mathcal{I}(bacbdcedae)$.

If $\left\vert V(G)\right\vert =6$ then Proposition \ref{deletion} tells us
that after replacing $G$ with a locally equivalent graph, we may presume that
$G$ has a vertex $v$ such that $G-v$ is prime. Then $G-v$ is locally
equivalent to $C_{5}$ \cite{Bec}, so after further local complementation of
$G$ (if necessary) we may presume that $G-v=C_{5}=\mathcal{I}(bacbdcedae)$.
Every proper subset of $\{a,b,c,d,e\}$ can be achieved as an interlacement
neighborhood of $v$ in a double occurrence word obtained by inserting two
appearances of $v$ into $bacbdcedae$: for instance $vbavcbdcedae$,
$bvacvbdcedae$, $vbacvbdcedae$, $bvacbdcvedae$, and $bvacbdvcedae$ provide $v$
with the interlacement neighborhoods $\{a,b\}$, $\{a,c\}$, $\{a,b,c\}$,
$\{a,b,d\}$, and $\{a,b,c,d\}$ respectively. The interlacement neighborhood
$\{a,b,c,d,e\}$ cannot be achieved in this way because the result is the wheel
graph $W_{5}$, which is not a Naji graph (as we saw in\ section 5).

We proceed using induction on $\left\vert V(G)\right\vert >6$. Observe that
Proposition \ref{circles} tells us that our inductive hypothesis is that
\emph{all} Naji graphs smaller than $G$ are circle graphs (not just the prime
ones). Proposition \ref{deletion} tells us that after local complementation,
we may presume that $G$ has a vertex $v$ such that $G-v$ is prime, and
similarly some graph $H$ locally equivalent to $G-v$ has a vertex $w$ such
that $H-w$ is prime. By applying the local complementations needed to obtain
$H$ from $G-v$ to $G$ before deleting $v$, we may presume simply that $G-v$
and $G-v-w$ are both prime.

We distinguish three cases.

Case 1. $G-w$ is prime.

With Corollary \ref{oneword}, the inductive hypothesis guarantees that up to
cyclic permutation and reversal, there is a unique double occurrence word $W$
with $\mathcal{I}(W)=G-v-w$, and there are unique locations to insert two
appearances of $v$ and two appearances of $w$ into $W$ so as to obtain double
occurrence words whose interlacement graphs are $G-w$ and $G-v$. Using cyclic
permutations, we may presume that if we insert both $v$ and $w$ into $W$ (at
the appropriate unique locations) then $v$ appears first, and if $v$ and $w$
are not interlaced then the second appearance of $v$ precedes the first
appearance of $w$. That is, $W=ABCD$ and if we let $G\Delta(vw)$ denote the
graph obtained from $G$ by reversing the adjacency status of $v$ and $w$, then
either $W^{\prime}=vAvBwCwD$ or $W^{\prime\prime}=vAwBvCwD$ correctly
describes $G$ or $G\Delta(vw)$ through interlacement. Of course if $W^{\prime
}$ or $W^{\prime\prime}$ correctly describes $G$, we are done. Otherwise,
either $W^{\prime}$ or $W^{\prime\prime}$ correctly describes $G\Delta(vw)$,
but does not correctly describe $G$.

Suppose $W^{\prime}$ correctly describes $G\Delta(vw)$, but does not succeed
in describing $G$. That is, if $e=vw$ then $e\in E(G)$ and $\mathcal{I}%
(W^{\prime})=G-e$. If $B$ or $D$ is empty then $v$ and $w$ appear
consecutively in $W^{\prime}$; we may simply interchange their consecutive
appearances to obtain a double occurrence word whose interlacement graph is
$G$. $A$ and $C$ cannot be empty as neither $v$ nor $w$ can be isolated, so we
proceed with the assumption that $A$, $B$, $C$ and $D$ are all nonempty. We
aim for a contradiction.

Let $\beta$ be any Naji solution for $G$. Theorem \ref{step1} tells us that
there is an orientation of $W$ corresponding to the Naji solution
$\beta|(G-v-w)$ of $G-v-w$, and there are extensions of this orientation to
$G-v$ and $G-w$ (i.e., in/out designations of the appearances of $v$ and $w$
in\ $W^{\prime}$) such that the resulting orientations of double occurrence
words correspond to the Naji solutions $\beta|(G-v)$ and $\beta|(G-w)$ of
$G-v$ and $G-w$. The only possible differences between $\beta$ and the Naji
solution $\beta^{\prime}$ of $\mathcal{I}(W^{\prime})$ corresponding to the
resulting orientation of $W^{\prime}$ involve the values $\beta(v,w)$ and
$\beta(w,v)$.

We claim that it is possible to choose $\beta$ so that $\beta=\beta^{\prime}$.
To verify the claim, suppose we have a $\beta$ with $\beta(v,w)\neq
\beta^{\prime}(v,w)$ and $\beta(w,v)=\beta^{\prime}(w,v)$. Consider
$\hat{\beta}=\beta+\delta_{G}(w)$, and let $(\hat{\beta})^{\prime}$ be the
Naji solution obtained by replacing $\beta$ with $\hat{\beta}$\ in the
preceding paragraph. As $\beta|(G-w)=\hat{\beta}|(G-w)$, $\beta$ and
$\hat{\beta}$ result in the same orientation of $W$, and the same orientation
of the word obtained from $W$ by inserting $v$. The restriction of $\delta
_{G}(w)$ to $G-v$ is $\delta_{G-v}(w)$, of course, and the effect of adding
$\delta_{G-v}(w)$ on the corresponding orientation is to interchange the
appearances of $w^{in}$ and $w^{out}$. Notice that interchanging the
appearances of $w^{in}$ and $w^{out}$ in $W^{\prime}=vAvBwCwD$ has the effect
that $(\hat{\beta})^{\prime}(v,w)=\beta^{\prime}(v,w)$ and $(\hat{\beta
})^{\prime}(w,v)\neq\beta^{\prime}(w,v)$. As $\hat{\beta}(v,w)\neq\beta(v,w)$
and $\hat{\beta}(w,v)\not =\beta(w,v)$, it follows that $\hat{\beta
}(v,w)=(\hat{\beta})^{\prime}(v,w)$ and $\hat{\beta}(w,v)=(\hat{\beta
})^{\prime}(w,v)$. Similarly, if $\beta(v,w)=\beta^{\prime}(v,w)$ and
$\beta(w,v)\not =\beta^{\prime}(w,v)$ then $\hat{\beta}=\beta+\delta_{G}(v)$
will have $\hat{\beta}(v,w)=(\hat{\beta})^{\prime}(v,w)$ and $\hat{\beta
}(w,v)=(\hat{\beta})^{\prime}(w,v)$; and if $\beta(v,w)\not =\beta^{\prime
}(v,w)$ and $\beta(w,v)\not =\beta^{\prime}(w,v)$ then $\hat{\beta}%
=\beta+\delta_{G}(v)+\delta_{G}(w)$ will have $\hat{\beta}(v,w)=(\hat{\beta
})^{\prime}(v,w)$ and $\hat{\beta}(w,v)=(\hat{\beta})^{\prime}(w,v)$.

Having verified the claim, we now know that $G$ and $G-e$ share a Naji
solution, with $e$ the edge $vw$. Corollary \ref{toggle2} tells us that
consequently, $G-e=\mathcal{I}(W^{\prime})$ has a Naji solution $\beta_{2}$
with $\beta_{2}(x,v)=\beta_{2}(x,w)$ $\forall x\notin\{v,w\}$. As $G$ is
prime, it has no cutpoint. Consequently $G-e$ is connected and we may cite
Proposition \ref{consecutive} to conclude that $v$ and $w$ appear
consecutively in $W$. But this contradicts the assumption that $B$ and $D$ are
both nonempty.

Suppose now that $W^{\prime\prime}=vAwBvCwD$ correctly describes $G\Delta
(vw)$, but does not succeed in describing $G$. That is, if $e=vw$ then $e\in
E(\mathcal{I}(W^{\prime\prime}))$ and $G=\mathcal{I}(W^{\prime\prime})-e$. If
any one of $A$, $B$, $C$, $D$ is empty then $v$ and $w$ appear consecutively
in $W^{\prime\prime}$; we may interchange consecutive appearances of $v$ and
$w$ to obtain a double occurrence word whose interlacement graph is $G$. We
proceed with the assumption that $A$, $B$, $C$ and $D$ are all nonempty, and
derive a contradiction.

As before, any Naji solution $\beta$ for $G$ leads to an orientation of
$W^{\prime\prime}$ with the property that the corresponding Naji solution
$\beta^{\prime\prime}$ of $\mathcal{I}(W^{\prime\prime})$ can only differ from
$\beta$ in the values $\beta(v,w)$ and $\beta(w,v)$. We claim again that it is
possible to choose $\beta$ so that $\beta^{\prime\prime}=\beta$. If we have a
$\beta$ such that $\beta(v,w)\neq\beta^{\prime\prime}(v,w)$ and $\beta
(w,v)=\beta^{\prime\prime}(w,v)$ then again, we consider $\hat{\beta}%
=\beta+\delta_{G}(w)$ and the Naji solution $(\hat{\beta})^{\prime\prime}$
obtained from $\hat{\beta}$ in the same way $\beta^{\prime\prime}$ is obtained
from $\beta$. And again, $\beta$ and $\hat{\beta}$ result in the same
orientation of the word obtained from $W$ by inserting $v$. But now when we
interchange $w^{in}$ and $w^{out}$ in $W^{\prime\prime}=vAwBvCwD$ we have
$(\hat{\beta})^{\prime\prime}(v,w)\not =\beta^{\prime\prime}(v,w)$ and
$(\hat{\beta})^{\prime\prime}(w,v)\neq\beta^{\prime\prime}(w,v)$. As
$\hat{\beta}(v,w)=\beta(v,w)$ and $\hat{\beta}(w,v)\not =\beta(w,v)$, though,
it follows again that $\hat{\beta}(v,w)=(\hat{\beta})^{\prime\prime}(v,w)$ and
$\hat{\beta}(w,v)=(\hat{\beta})^{\prime\prime}(w,v)$. As before, the
possibility that $\beta(w,v)\neq\beta^{\prime\prime}(w,v)$ is handled by using
$\delta_{G}(v)$. The claim allows us to use Corollary \ref{toggle2} and
Proposition \ref{consecutive} to derive a contradiction.

Case 2. $G-w$ is not connected.

This case cannot occur as the prime graph $G$ cannot have a cutpoint.

Case 3. $G-w$ is connected and not prime. (This is the most delicate case.)

Let $(X,Y)$ be a split of $G-w$ with $v\in X$. If $\left\vert X\right\vert >2$
then $(X-v,Y)$ is a split of $G-w-v$, an impossibility as $G-w-v$ is prime.
Consequently $\left\vert X\right\vert =2$; let $X=\{v,x\}$. If $v$ and $x$ are
not adjacent we may pick any $y\in N(v)=N(x)$, and replace $G$ with the local
complement $G^{y}$. If $v$ and $x$ are adjacent and $N(v)\cap Y\neq
\varnothing$ we may replace $G$ with the local complement $G^{x}$. After these
replacements we see that we may assume that the degree of $v$ in $G-w$ is 1.
As $G$ is prime, the degree of $v$ in $G$ cannot be 1; hence $vw\in E(G)$, and
$N(v)=\{w,x\}$.

Replacing $G$ with $G^{v}$ if necessary, we may presume that $wx\in E(G)$.
According to Corollary \ref{solutions}, $G$ has a Naji solution $\beta$ with
$\beta(w,v)=\beta(x,v)=0$ and $\beta(u,v)=1$ $\forall u\notin\{v,w,x\}$. We
have no further need for the hypothesis that $G-v-w$ is prime, so it does no
harm to assume that $\beta(w,x)=1$; if $\beta(w,x)=0$, we simply interchange
the names of $w$ and $x$. Our job, then, is to produce a double occurrence
word whose interlacement graph is $G$, under the assumptions that $G$ is a
prime Naji graph, $G-v$ is a prime circle graph, $N(v)=\{w,x\}$,
$\beta(w,v)=\beta(x,v)=0$, $\beta(w,x)=1$ and $\beta(u,v)=1$ $\forall
u\notin\{v,w,x\}$.

According to Corollary \ref{oneword}, up to cyclic permutation and reversal
there is a unique double occurrence word $W$ whose interlacement graph is
$G-v$. Theorem \ref{step1} tells us that $W$ can be oriented so that the
corresponding Naji solution is the restriction $\beta|(G-v)$. Cyclically
permute $W$ so the first letter is $w^{in}$.\ As $w$ and $x$ are interlaced
and $\beta(w,x)=1$,%
\[
W=w^{in}Ax^{in}Bw^{out}Cx^{out}D
\]
for some subwords $A$, $B$, $C$ and $D$. We will use pairs of letters to
designate subsets of $V(G)$ in the natural way: $AB$ denotes the set of
vertices that appear once in $A$ and once in $B$, $CC$ denotes the set of
vertices that appear twice and $C$ and so on. Our aim is to prove that there
must be locations in $W$ where we can place two appearances of $v$ so that $v$
is interlaced with $w$ and $x$, but not with any other vertex. If any of $A$,
$B$, $C$, $D$ is empty then $w$ and $x$ appear consecutively in $W$ and we can
accomplish our aim by replacing a subword $wx$ or $xw$ with $vwxv$ or $vxwv$.
Consequently we may proceed with the assumption that none of $A$, $B$, $C$,
$D$ is empty.

The rest of the argument is a sequence of claims. During the discussion of the
claims we will often use the fact that if $y,z\in V(G)-\{v,w,x\}$ the Naji
equations require $\beta(v,y)=\beta(v,z)$ if $yz$ is an edge, or there is a
path from $y$ to $z$ in $G-v-w-x$.

Claim 1. $AC=\varnothing$.

proof: Suppose $a\in AC$. If $a^{in}$ appears in $A$ then $\beta(a,w)=0$ and
$\beta(a,x)=1$; if $a^{in}$ appears in $C$ these values are reversed. Either
way, $\beta(a,w)\neq\beta(a,x)$. As $av\notin E(G)$ and $aw,ax,vw,vx\in E(G)$,
the Naji equations of $G$ require%
\begin{align*}
\beta(a,w)+\beta(v,w)+\beta(a,v)+\beta(v,a)  &  =1\text{ and}\\
\beta(a,x)+\beta(v,x)+\beta(a,v)+\beta(v,a)  &  =1\text{.}%
\end{align*}
It cannot be that both equations hold as the first terms are unequal and the
other terms are all equal. {\ \rule{0.5em}{0.5em}}

Claim 2. If $y\in AA\cup AB\cup BB\cup BC\cup CC\cup CD\cup DD$ then $y^{in}$
precedes $y^{out}$, and if $y\in AD$ then $y^{out}$ precedes $y^{in}$.

proof: If $y\in AA\cup AB\cup BB\cup CC\cup CD\cup DD$ then $y$ is not
interlaced with $w$, so the Naji equations require $\beta(y,w)=\beta(y,v)=1$;
this in turn requires that $y^{in}$ precede $y^{out}$. If $y\in BC$ then $y$
is not interlaced with $x$, so $\beta(y,x)=\beta(y,v)=1$ and again this
requires that $y^{in}$ precede $y^{out}$. If $y\in AD$, instead, then
$\beta(y,x)=\beta(y,v)=1$ implies that $y^{in}$ appears in $D$.
{\ \rule{0.5em}{0.5em}}

Claim 3. If $a\in AA\cup AB\cup AD$ then $\beta(v,a)=0$.

proof: If $a\in AB$ then claim 2 implies that $\beta(a,x)=1$. As $a\in
N(x)-N(v)$, the Naji equations require%
\[
1=\beta(a,x)+\beta(v,x)+\beta(a,v)+\beta(v,a)=1+1+1+\beta(v,a)\text{.}%
\]
Similarly, if $a\in AD$ then $\beta(a,w)=1$ and $a\in N(w)-N(v)$, so the Naji
equations require%
\[
1=\beta(a,w)+\beta(v,w)+\beta(a,v)+\beta(v,a)=1+1+1+\beta(v,a)\text{.}%
\]
If $a\in AA$ then as $G-v$ is connected, there is a shortest path in $G-v$
from $a$ to some vertex not in $AA$. As $AC=\varnothing$ by claim 1, the
definition of interlacement makes it clear that the last vertex on this path
is in $AB$ or $AD$; as $\beta(v,y)=0$ for every such vertex $y$, and no vertex
on the path neighbors $v$, the Naji equations require that $\beta(v,a)=0$.
{\ \rule{0.5em}{0.5em}}

Claim 4. If $c\in BC\cup CC\cup CD$ then $\beta(v,c)=1$.

proof: The proof is closely analogous to that of claim 3.
{\ \rule{0.5em}{0.5em}}

Claim 5. In the $B$ portion of $W$, all vertices from $AB$ precede all
vertices from $BC$.

proof: If $a\in AB$ and $c\in BC$ were interlaced, a Naji equation would
require that $\beta(v,a)=\beta(v,c)$; but this would contradict claims 3 and
4. {\ \rule{0.5em}{0.5em}}

Claim 6. In the $D$ portion of $W$, all vertices from $CD$ precede all
vertices from $AD$.

proof: If $c\in CD$ and $a\in AD$ were interlaced, a Naji equation would
require that $\beta(v,a)=\beta(v,c)$; but this would contradict claims 3 and
4. {\ \rule{0.5em}{0.5em}}

Observe that claim 5 tells us we can partition the $B$ portion of $W$ as
$B_{0}B_{1}B_{2}$ in such a way that all vertices from $AB$ appear in $B_{0}$,
the last letter in $B_{0}$ is a vertex from $AB$, all vertices from $BC$
appear in $B_{2}$ and the first letter in $B_{2}$ is a vertex from $BC$. Claim
6 tells us that we can partition the $D$ portion of $W$ as $D_{0}D_{1}D_{2}$
in such a way that all vertices from $CD$ appear in $D_{0}$, the last letter
in $D_{0}$ is a vertex from $CD$, all vertices from $AD$ appear in $D_{2}$ and
the first letter in $D_{2}$ is a vertex from $AD$. (Some of the subwords
$B_{i},D_{i}$ may be empty.) Consequently we have%
\[
W=w^{in}Ax^{in}B_{0}B_{1}B_{2}w^{out}Cx^{out}D_{0}D_{1}D_{2}\text{.}%
\]

Claim 7. If $y$ appears in $B_{0}$ or $D_{2}$ then $\beta(v,y)=0$, and if $y$
appears in $B_{2}$ or $D_{0}$ then $\beta(v,y)=1$.

proof: Consider a vertex $y$ that appears in $B_{0}$. If $y\in AB$ then
$\beta(v,y)=0$ by claim 3. If $y\in BD$ then $y$ is interlaced with the vertex
$a\in AB$ that appears at the end of $B_{0}$, so $\beta(v,y)=\beta(v,a)=0$.
The same argument applies if $y\in BB$ appears only once in $B_{0}$. If $y\in
BB$ appears twice in $B_{0}$, then as $G$ is connected, some path must lead
from $y$ to a vertex $z$ that appears only once in $B_{0}$. Consider such a
path of shortest length; then all the vertices on the path before $z$ are,
like $y$, elements of $BB$ that appear twice in $B_{0}$. Then $\beta(v,z)=0$
by the earlier parts of the argument, and the Naji equations require that
$\beta(v,y)=\beta(v,z)$. Similar arguments apply in $B_{2}$, $D_{0}$ and
$D_{2}$. {\ \rule{0.5em}{0.5em}}

Notice that claim 7 implies that vertices from $B_{0}$ or $D_{2}$ cannot
appear in $B_{2}$ or $D_{0}$, and vice versa. Consequently if $B_{1}$ and
$D_{1}$ are both empty then the word%
\[
w^{in}Ax^{in}B_{0}vB_{2}w^{out}Cx^{out}D_{0}vD_{2}%
\]
has $G$ as its interlacement graph. Our aim is to show that if $B_{1}$ and
$D_{1}$ are not empty, they have \textquotedblleft centers\textquotedblright%
\ where we can insert the desired appearances of $v$.

To locate these centers we repartition $B$ and $D$. Let $B=B^{1}B^{2}%
B^{3}...B^{k}$ in such a way that each $B^{i}$ is nonempty and the value of
$\beta(v,-)$ is constant on each $B^{i}$, with $\beta(v,-)$ changing when we
pass from $B^{i}$ to $B^{i+1}$. The fact that $B$ is nonempty tells us that
$k\geq1$. Claim 7 tells us that $B^{1}$ contains $B_{0}$ and $B^{k}$ contains
$B_{2}$. (N.b. $B_{0}$ or $B_{2}$ might be empty.) Partition $D$ as
$D^{1}...D^{\ell}$ in a similar way.

Claim 8. For each $i$, there is a vertex that appears precisely once in
$B^{i}$. \ Similarly, in each $D^{j}$ some vertex appears exactly once.

proof: If every vertex that appears in $B^{i}$ appears twice in $B^{i}$, then
no vertex that appears in $B^{i}$ is interlaced with any vertex that does not
appear in $B^{i}$. This is impossible, as $B^{i}$ is not empty and $G$ is
connected. The same observation applies to $D^{j}$. {\ \rule{0.5em}{0.5em}}

Claim 9. If $i\neq j$ then no vertex appears in both $B^{i}$ and $B^{j}$, and
no vertex appears in both $D^{i}$ and $D^{j}$.

proof: Suppose $i<j$, a vertex $y$ appears in both $B^{i}$ and $B^{j}$, and
$j-i$ is as small as possible. Claim 8 tells us that there is a vertex $z$,
which appears precisely once in $B^{i+1}$. The minimality of $j-i$ guarantees
that the other appearance of $z$ is outside the subword $B^{i}B^{i+1}...B^{j}%
$, so $y$ and $z$ are interlaced. The Naji equations then require
$\beta(v,y)=\beta(v,z)$, contradicting the definition of $B^{1}B^{2}%
B^{3}...B^{k}$, which guarantees $\beta(v,y)\neq\beta(v,z)$. The same argument
applies to $D^{i}$ and $D^{j}$. {\ \rule{0.5em}{0.5em}}

Claim 10. Suppose $y$ appears precisely once in $B^{i}$. If $i=1$ then $y\in
AB\cup BD$, if $1<i<k$ then $y\in BD$, and if $i=k$ then $y\in BC\cup BD$.
Similarly, if $z$ appears precisely once in $D^{j}$ then $j=1$ implies $z\in
BD\cup CD$, $1<j<\ell$ implies $z\in BD$, and $j=\ell$ implies $z\in AD\cup
BD$. {\ \rule{0.5em}{0.5em}}

proof: Claim 9 tells us that $y\notin BB$, as every element of $BB$ appears
twice in the same one of $B^{1},...,B^{k}$. The assertion regarding $B$
follows because all appearances in $B$ of elements of $AB$ occur in $B_{0}$,
which is a subword of $B^{1}$; and all appearances in $B$ of elements of $BC$
occur in $B_{2}$, which is a subword of $B^{k}$. The assertion regarding $D$
is verified in the same way. {\ \rule{0.5em}{0.5em}}

Claim 11. Neither $k>2$ nor $\ell>2$ is possible.

proof: Suppose $k>2$. Let $Y^{1}$, $Y^{2}$, $Y^{3}\subseteq V(G)$ be the
subsets consisting of vertices that appear precisely once in $B^{1}$, $B^{2}$
and $B^{3}$ respectively. Claims 8 and 9 tell us that $Y^{1}$, $Y^{2}$, and
$Y^{3}$ are nonempty and pairwise disjoint, and claim 10 guarantees that
$Y^{2}\subseteq BD$.

Subclaim 11a. All the second appearances of elements of $Y^{2}$ appear in the
same $D^{j}$.

proof: If $y,y^{\prime}\in Y^{2}$ appear in $D^{i}$ and $D^{j}$ with $i<j$
then consider a vertex $z$ that appears once in $D^{i+1}$; we have
$\beta(v,z)\neq\beta(v,y)=\beta(v,y^{\prime})$. Claim 10 tells us that $z\in
BD$, so the other appearance of $z$ is in $B$; if $z$ appears in $B^{1}$ it is
interlaced with $y^{\prime}$, and if $z$ does not appear in $B^{1}$ then it is
interlaced with $y$. Either way we have a contradiction as the Naji equations
require that interlaced elements of $BD$ have the same $\beta(v,-)$ value.
{\ \rule{0.5em}{0.5em}}

We now let $\tau$ denote the index of the particular $D^{j}$ that includes all
the second appearances of elements of $Y^{2}$. No element of $Y^{2}$ can be
interlaced with an element of $Y^{1}\cup Y^{3}$, as the value of $\beta(v,-)$
on $Y^{2}$ is different from the value on $Y^{1}\cup Y^{3}$. Consequently if
$y\in Y^{1}$ then the other appearance of $y$ must either occur in $D$ after
$D^{\tau}$ (if $y\in BD$) or in $A$ (if $y\in AD$). Also, if $y\in Y^{3}$ then
the other appearance of $y$ must occur before $D^{\tau}$, either in $D$ (if
$y\in BD$) or in $C$ (if $y\in BC$).

Subclaim 11b. Every vertex that appears precisely once in $D^{\tau}$ also
appears in $Y^{2}$.

proof: Suppose $d$ appears once in $D^{\tau}$ and the other appearance of $d$
is not in $Y^{2}$. Claim 9 tells us that the other appearance of $d$ is not in
$D$. If the other appearance of $d$ is in $A$, then $d\in AD$ so
$\beta(v,d)=0$ by claim 3. Then $\beta(v,y)=0$ $\forall y\in Y^{2}$ and
$\beta(v,y^{\prime})=1$ $\forall y^{\prime}\in Y^{1}$; consequently no vertex
that appears in $Y^{1}$ is an element of $AD$, so every vertex $y^{\prime}$
that appears in $Y^{1}$ is an element of $BD$. According to the paragraph
before the statement of this subclaim, every $y^{\prime}$ that appears in
$Y^{1}$ appears in $D$ after $D^{\tau}$; it follows that $d$ is interlaced
with every such $y^{\prime}$. But this is impossible because the $\beta(v,-)$
values do not match. A similar line of reasoning applies if the other
appearance of $d$ is in $C$: $\beta(v,d)=1$ by claim 4, so $\beta(v,y)=1$
$\forall y\in Y^{2}$ and $\beta(v,y^{\prime})=0$ $\forall y^{\prime}\in Y^{3}%
$; hence no element of $Y^{3}$ lies in $BC$, so every element of $Y^{3}$ lies
in $BD$. It follows that every element of $Y^{3}$ appears in $D$ before
$D^{\tau}$, by the paragraph before the statement of this subclaim; but then
every such element is interlaced with $d$, and again the $\beta(v,-)$ values
prohibit this. Consequently $d$ must appear in $B$. But then $d$ appears after
$Y^{3}$ in $B$ and also after $Y^{3}$ in $D$, an impossibility because $d$
cannot be interlaced with any element of $Y^{3}$. {\ \rule{0.5em}{0.5em}}

Now consider the subwords $Y^{2}$ and $D^{\tau}$ of $W$. Subclaims 11a and 11b
tell us that if $X$ is the set of vertices of $G-v$ that appear outside
$Y^{2}D^{\tau}$ and $Y$ is the set of vertices that appear within
$Y^{2}D^{\tau}$, then $X\cap Y=\varnothing$. $G-v$ is prime, so $(X,Y)$ cannot
be a split. As $X$ contains $w$ and $x$ along with all the vertices that
appear in $Y^{1}$ and $Y^{3}$, $\left\vert X\right\vert >2$; hence $\left\vert
Y\right\vert =1$. That is, $Y^{2}$ and $D^{\tau}$ are both of length 1, and
mention the same vertex.

Subclaim 11c. No vertex of $G-v-w-x$ is interlaced with the vertex $y$ that
appears in $Y^{2}$ and $D^{\tau}$.

proof: Suppose $z$ is interlaced with $y$. Then $z\notin AA\cup CC$. If $z\in
BB\cup DD$ then claim 9 tells us that $z$ appears twice in $Y^{2}$ or
$D^{\tau}$, an impossibility as no vertex other than $y$ appears in either
$Y^{2}$ or $D^{\tau}$. If $z\in AB$ then $z\in B^{1}$, so $\beta(v,z)\neq
\beta(v,y)$; this is not possible if $yz\in E(G)$. If $z\in AD$ then $z$
appears in $A$ before the appearance of $y$ in $B$, so $z$ also appears before
$y$ in $D$; as the vertices of $AD$ all appear in $D^{\ell}$ it follows that
$\tau=\ell$. But then both $y$ and $z$ appear in $D^{\tau}$, contradicting the
fact that only $y$ appears in $D^{\tau}$. If $z\in BC$ then the appearance of
$z$ in $C$ precedes the appearance of $y$ in $D$, so the appearance of $z$ in
$B$ must precede the appearance of $y$ in $B$; but then $z$ appears in $B^{1}$
so $\beta(v,z)\neq\beta(v,y)$, an impossibility if $y$ and $z$ are neighbors
in $G$. If $z\in BD$ then $\beta(v,z)=\beta(v,y)$, so $z$ appears in some
$B^{2i}$ with $i>1$. Consider a vertex $b$ that appears once in $B^{3}$. To
avoid being interlaced with $z$, $b$ must appear after $z$ in $D$. To avoid
being interlaced with $y$, $b$ must appear before $y$ in $D$. Hence $z$
appears before $y$ in $D$; but this cannot be the case as $z$ appears after
$y$ in $B$ and $yz\in E(G)$. The only remaining possibility is $z\in CD$. Such
a $z$ would have to appear in or after $D^{\tau+2}$, in order to be interlaced
with $y$. Suppose $b$ appears once in $B_{3}$; then $b$ must appear in $D$
before $D^{\tau}$, to avoid being interlaced with $y$. Consequently $b$ and
$z$ are interlaced, an impossibility as $\beta(v,z)=\beta(v,y)\neq\beta(v,b)$.
{\ \rule{0.5em}{0.5em}}

Subclaim 11c completes the proof that $k>2$ is impossible, for it implies that
$N(y)=\{w,x\}$; this in turn implies that $\{v,y\}$ is a split of $G$.

The assertion that $\ell>2$ is impossible can be proven in the same way, so we
are done with claim 11. {\ \rule{0.5em}{0.5em}}

Claim 12. At least one of $AB,AD$ is not empty, and at least one of $BC,CD$ is
not empty. Consequently, $k+\ell\geq2$.

proof: If $AB=AD=\varnothing$ then as $A$ is not empty, it must be that
$AA\not =\varnothing$. But no edge of $G$ can connect a vertex of $AA$ to a
vertex outside $AA$, contradicting the fact that $G$ is connected. Hence at
least one of $AB,AD$ is not empty. If $AB\neq\varnothing$ then $k\geq1$, as
the vertices of $AB$ all appear in $B^{1}$; and if $AD\neq\varnothing$ then
$\ell\geq1$, as the vertices of $AD$ all appear in $D^{\ell}$. Similarly, at
least one of $BC,CD$ is not empty, so at least one of $B^{k},D^{1}$ is not
empty. Finally, no vertex of $AB\cup AD$ can appear in the same set $B^{i}$ or
$D^{j}$ as a vertex of $BC\cup CD$, because the $\beta(v,-)$ values do not
match. {\ \rule{0.5em}{0.5em}}

At this point we change notation slightly. If $k=2$ then we let $B(i)=B^{i}$
for $i\in\{1,2\}$. If $k=1$ and the value of $\beta(v,-)$ on $B^{1}$ is 0, we
let $B(1)=B^{1}$, and we let $B(2)$ denote the empty word. If $k=1$ and the
value of $\beta(v,-)$ on $B^{1}$ is 1, we let $B(2)=B^{1}$ and we let $B(1)$
denote the empty word. Similarly, we define $D(1)=D^{1}$ and $D(2)=D^{2}$ if
$\ell=2$, and if $\ell=1$ we define $D(1)$ and $D(2)$ so that one is empty,
the other is $D^{1}$ and the value of $\beta(v,-)$ on $D(i)$ is $i$ (mod 2).
We now have
\[
W=w^{in}Ax^{in}B(1)B(2)w^{out}Cx^{out}D(1)D(2)\text{,}%
\]
where up to two of $B(1),B(2),D(1),D(2)$ may be empty.

Claim 13. The value of $\beta(v,-)$ is 0 on $B(1)$ and $D(2)$, and 1 on $B(2)$
and $D(1)$.

proof: If $AB\neq\varnothing$ then the claim is true for $B(1)$, as every
vertex $a\in AB$ appears in $B(1)$ and has $\beta(v,a)=0$. Necessarily then
the claim is also true for $B(2)$, as the values of of $\beta(v,-)$ on $B(1)$
and $B(2)$ are different. Similarly, if $AD\neq\varnothing$ then the claim is
true for $D(1)$ and $D(2)$, as every $a\in AD$ appears in $D(2)$ and has
$\beta(v,a)=0$. The same reasoning shows that the claim holds in $B$ if
$BC\neq\varnothing$, and the claim holds in $D$ if $CD\neq\varnothing$.

Claim 12 now assures us that claim 13 holds in at least one of $B$ and $D$;
suppose it holds in $B$. If the claim does not hold in $D$ then $D(1)=D^{1}$
and $D(2)=D^{2}$ are both nonempty, the value of $\beta(v,-)$ on $D(1)$ is 0,
and the value of $\beta(v,-)$ on $D(2)$ is 1. Let $d_{1}$ and $d_{2}$ be
vertices that appear precisely once in $D(1)$ and $D(2)$, respectively. Then
$d_{1},d_{2}\notin DD$, of course, and according to claims 3 and 4, the values
of $\beta(v,d_{1})$ and $\beta(v,d_{2})$ indicate that $d_{1}\notin CD$ and
$d_{2}\notin AD$. Consequently $d_{1}\in AD\cup BD$ and $d_{2}\in BD\cup CD$.
As $\beta(v,d_{1})\neq\beta(v,d_{2})$, $d_{1}$ and $d_{2}$ are not interlaced;
$d_{1}$ precedes $d_{2}$ in $D$, so $d_{2}$ must precede $d_{1}$ outside $D$.
This is impossible if $d_{1}\in AD$ or $d_{2}\in CD$, so it must be that
$d_{1},d_{2}\in BD$. But then the values of $\beta(v,d_{i})$ indicate that
$d_{1}$ appears in $B(1)$ and $d_{2}$ appears in $B(2)$, so $d_{2}$ does not
precede $d_{1}$ outside $D$. By contradiction, we conclude that if claim 13
holds in $B$ it also holds in $D$. The converse is justified in the same way.
{\ \rule{0.5em}{0.5em}}

Claim 14. $G$ is the interlacement graph of the double occurrence word%
\[
W^{\prime}=w^{in}Ax^{in}B(1)vB(2)w^{out}Cx^{out}D(1)vD(2)\text{.}%
\]

proof: A vertex $a$ that appears in $A$ has $\beta(v,a)=0$, by claims 1 and 3,
so it may appear twice in $A$, or once in $A$ and once in $B(1)$, or once in
$A$ and once in $D(2)$. In any case it is not interlaced with $v$ in
$W^{\prime}$. Similarly a vertex that appears in $C$ may appear again in $C$,
or appear in $B(2)$ or $D(1)$; in any case it is not interlaced with $v$. A
vertex\ $b$ that appears in $B$ and $D$ must appear either in $B(1)$ and
$D(2)$ (if the value of $\beta(v,b)$ is 0) or in $B(2)$ and $D(1)$ (if the
value of $\beta(v,b)$ is 1); again, neither case allows it to be interlaced
with $v$. Finally, an element of $BB$ or $DD$ is not interlaced with $v$.
Consequently $w$ and $x$ are the only vertices interlaced with $v$ in
$W^{\prime}$.
\end{proof}

\section{Bipartite graphs}

Bipartite circle graphs are special for two reasons, both connected with
planarity. One special property is geometric: bipartite circle graphs
correspond to planar 4-regular graphs \cite{RR}. (All circle graphs correspond
to 4-regular graphs, as a double occurrence word naturally gives rise to an
Euler circuit in a 4-regular graph.) Another special property is matroidal: a
bipartite graph with adjacency matrix $A$ is a circle graph if and only if the
binary matroid represented by $%
\begin{pmatrix}
I & A
\end{pmatrix}
$ is planar \cite{Fi}; here $I$ is an identity matrix. (This matroid is the
direct sum of a pair of mutually dual matroids, so it is planar if and only if
it is graphic or cographic.)

At the end of \cite{GG}, Geelen and Gerards deduce an algebraic
characterization of planar matroids from their characterization of graphic
matroids. The following theorem provides a bridge between their result and
Naji's theorem.

\begin{theorem}
Let $G$ be a bipartite graph with vertex classes $V_{1}$ and $V_{2}$. Then $G$
is a circle graph if and only if this system of equations has a solution over
$GF(2)$.

(a) If $v,w,x$ are three different elements of the same vertex class and
$N(v)\cap N(w)\not \subseteq N(x)$, then $\beta(x,v)=\beta(x,w)$.

(b) If $v,w,x$ are three different elements of the same vertex class and
$N(v)\cap N(w)\cap N(x)\neq\varnothing$, then%
\[
\beta(v,w)+\beta(w,v)+\beta(v,x)+\beta(x,v)+\beta(w,x)+\beta(x,w)=1.
\]

\end{theorem}

\begin{proof}
The equations mentioned in the statement follow directly from the Naji
equations. For (a), note that if $y\in N(v)\cap N(w)-N(x)$, then the Naji
equations require $\beta(x,v)=\beta(x,y)$ and $\beta(x,y)=\beta(x,w)$. For
(b), note that if $y\in N(v)\cap N(w)\cap N(x)$ we may add together the
following Naji equations.%
\begin{align*}
\beta(y,v)+\beta(y,w)+\beta(v,w)+\beta(w,v)  &  =1\\
\beta(y,v)+\beta(y,x)+\beta(v,x)+\beta(x,v)  &  =1\\
\beta(y,w)+\beta(y,x)+\beta(w,x)+\beta(x,w)  &  =1
\end{align*}

For the converse, suppose the equations mentioned in the statement of this
theorem have a solution $\beta$. Note that the equations require only that
$\beta(v,w)$ be defined when $v$ and $w$ are elements of the same vertex
class. In order to build a Naji solution we must define values of $\beta(v,w)$
when $v$ and $w$ are not elements of the same vertex class. According to
Proposition \ref{components}, we may presume that $G$ is connected.

Suppose $v\in V_{1}$, $w\in V_{2}$ and $vw\notin E(G)$. As $G$ is connected,
there are $v^{\prime}\in V_{1}$ and $w^{\prime}\in V_{2}$ such that
$v^{\prime}w,vw^{\prime}\in E(G)$. Define $\beta(v,w)=\beta(v,v^{\prime})$ and
$\beta(w,v)=\beta(w,w^{\prime})$. The equations of part (a) of the statement
guarantee that these values are well defined. Moreover, these definitions
satisfy all the Naji equations listed under (b) in Definition \ref{najieq}.

We index the elements of $V_{1}$ and $V_{2}$, $V_{1}=\{v_{1},...,v_{a}\}$ and
$V_{2}=\{w_{1},...,w_{b}\}$, in such a way that $v_{1}w_{1}\in E(G)$ and for
$i>1$,%
\[
N(v_{i})\cap\{w_{1},...,w_{i-1}\}\neq\varnothing\neq N(w_{i})\cap
\{v_{1},...,v_{i}\}\text{.}%
\]
One way to construct such an indexing recursively is to find a leaf $v$ of a
spanning tree $T$ for $G$, find an indexing of the specified type for $T-v$,
and then list $v$ as $v_{a}$ or $w_{b}$ according to whether $v\in V_{1}$ or
$v\in V_{2}$.

To define the values of $\beta(v,w)$ with $vw\in E(G)$, begin by defining
$\beta(v_{1},w_{1})$ $=0$ and $\beta(w_{1},v_{1})$ $=1$. If $i>1$ and
$v_{1}w_{i}\in E(G)$, define $\beta(v_{1},w_{i})=\beta(v_{1},w_{1}%
)+\beta(w_{i},w_{1})+\beta(w_{1},w_{i})+1$ and $\beta(w_{i},v_{1}%
)=1+\beta(v_{1},w_{i})$. Interchange the letters $v$ and $w$ to define
$\beta(w_{1},v_{i})$ and $\beta(v_{i},w_{1})$ if $i>1$ and $w_{1}v_{i}\in
E(G)$. It is easy to check that all the Naji equations involving $v_{1}$ or
$w_{1}$ are satisfied. Suppose $i_{0}>1$ and all values of $\beta(v_{i}%
,w_{j})$ and $\beta(w_{j},v_{i})$ have been defined when $i<i_{0}$ or
$j<i_{0}$, in such a way that all Naji equations are satisfied. By hypothesis,
$v_{i_{0}}$ has a neighbor $w_{j_{0}}$ with $j_{0}<i_{0}$. If\ $j>i_{0}$ and
$v_{i_{0}}w_{j}\in E(G)$, define $\beta(v_{i_{0}},w_{j})=\beta(v_{i_{0}%
},w_{j_{0}})+\beta(w_{j},w_{j_{0}})+\beta(w_{j_{0}},w_{j})+1$; equation (b) of
the statement guarantees that this definition is independent of the choice of
a particular $w_{j_{0}}\in N(v_{i_{0}})$. Also define $\beta(w_{j},v_{i_{0}%
})=1+\beta(v_{i_{0}},w_{j})$. The equations of the statement imply that all
Naji equations involving $v_{i_{0}}$ are satisfied. The values of
$\beta(w_{i_{0}},v_{j})$ and $\beta(v_{j},w_{i_{0}})$ when $j>i_{0}$ and
$v_{j}w_{i_{0}}\in E(G)$ are defined in the same way, mutatis mutandi.
\end{proof}

We should mention that a different way to reformulate Naji's theorem for
bipartite graphs was given by Bouchet \cite{Bbip}.

\section{Permutation graphs}

Here is a familiar definition, discussed for instance by Golumbic
\cite[Chapter 7]{Go}.

\begin{definition}
Let $\pi$ be a permutation of $\{1,...,n\}$. Then the corresponding
\emph{permutation graph} has vertices $1,...,n$, with an edge $ij$ whenever
$i<j$ and $\pi(i)>\pi(j)$.
\end{definition}

Naji's theorem leads to the following algebraic characterization of
permutation graphs.

\begin{theorem}
A simple graph $G$ is a permutation graph if and only if this system of
equations has a solution over $GF(2)$.

(a) If $v$ and $w$ are two distinct vertices then $\beta(v,w)+\beta(w,v)=1$.

(b) If $v,w,x$ are three distinct vertices such that $vw,vx\notin E(G)$ and
$wx\in E(G)$ then $\beta(v,w)+\beta(v,x)=0$.

(c) If $v,w,x$ are three distinct vertices such that $vw,vx\in E(G)$ and
$wx\not \in E(G)$ then $\beta(v,w)+\beta(v,x)=0$.
\end{theorem}

\begin{proof}
Suppose $G$ is the permutation graph corresponding to the permutation $\pi$.
Let $W$ be the oriented double occurrence word
\[
1^{in}...n^{in}\pi(n)^{out}...\pi(1)^{out}.
\]
Then the interlacement graph $\mathcal{I}(W)$ is $G$, and the Naji solution
$\beta$ corresponding to $W$ has the property that $\beta(w,x)\neq\beta(x,w)$
$\forall w\neq x\in V(G)$. Consequently $\beta$ satisfies the equations of the statement.

For the converse, suppose $\beta$ satisfies the equations of the statement,
and let $G+z$ be the graph obtained from $G$ by adjoining a new vertex $z$
adjacent to all the vertices of $G$. Extend $\beta$ by defining $\beta
(-,z)\equiv0$ and $\beta(z,-)\equiv1$. Then the extended $\beta$ is a Naji
solution for $G+z$, so $G+z$ is a circle graph. (The fact that $G$ is a
permutation graph if and only if $G+z$ is a circle graph is well known, and
easily proven without Naji's theorem; see for instance \cite[Exercise
11.12]{Go}.) If $zW_{1}zW_{2}$ is a double occurrence word\ with interlacement
graph $G+z$ then each vertex of $G$ must appear once in each $W_{i}$, in order
to be interlaced with $z$. Consequently $W_{1}$ and $W_{2}$\ provide a
permutation representation of $G$.
\end{proof}

Another way to say the same thing is this: an $n$-vertex simple graph is a
permutation graph if and only if it shares a Naji solution with $K_{n}$.

\section{Conclusion}

We finish the paper by mentioning several promising directions for further
research into the significance of Naji's theorem.

The first is suggested by\ a comment of Geelen, which was mentioned in the introduction:

\begin{problem}
Extend Theorem \ref{step1} and Corollary \ref{moresplit} to a precise
relationship between $\mathcal{B}(G)$ and the split decomposition of $G$.
\end{problem}

A researcher interested in this problem will appreciate the thorough
discussion of circle graphs and split decompositions given by Courcelle
\cite{C}.

A second problem was suggested by an anonymous reader.

\begin{problem}
Find special forms of the Naji equations that characterize other special types
of circle graphs, in addition to bipartite circle graphs and permutation graphs.
\end{problem}

There are many candidates for such \textquotedblleft special\textquotedblright%
\ circle graphs, such as diamond-free circle graphs \cite{DGR}, distance
hereditary graphs \cite{BM} and linear domino circle graphs \cite{BDGS}.

A third problem is suggested by the derivation of Naji's theorem from
Bouchet's circle graph obstructions theorem \cite{Bco} given by Gasse \cite{G}.

\begin{problem}
Derive Bouchet's obstructions theorem from Naji's theorem.
\end{problem}

In addition to Bouchet's characterization by obstructions, there are also
characterizations of circle graphs using binary matroids \cite{BT2},
delta-matroids \cite{GeelenPhD} and monadic second-order logic \cite{C}. The
first two of these characterizations involve the field $GF(2)$, and the third
involves the even cardinality set predicate. As the Naji equations are defined
over $GF(2)$, it seems reasonable to guess that Naji's equations might be
connected to them in some way.

\begin{problem}
Relate Naji's theorem to other characterizations of circle graphs.
\end{problem}

\end{document}